\documentclass[reqno]{amsart}

\makeatletter
\@namedef{subjclassname@2020}{%
  \textup{2020} Mathematics Subject Classification}
\makeatother

\usepackage{amsmath}
\usepackage{amssymb}
\usepackage{enumerate}

\newtheorem{Theorem}{Theorem}[section]
\newtheorem{Lemma}[Theorem]{Lemma}
\newtheorem{Corollary}[Theorem]{Corollary}
\newtheorem{Proposition}[Theorem]{Proposition}

\newcommand{\thref}[1]{Theorem \ref{#1}}
\newcommand{\leref}[1]{Lemma \ref{#1}}
\newcommand{\reref}[1]{Remark \ref{#1}}
\newcommand{\seref}[1]{Section \ref{#1}}

\theoremstyle{definition}
\newtheorem{Remark}[Theorem]{Remark}
\newtheorem{Example}[Theorem]{Example}
\newtheorem*{Remark*}{Remark}

\numberwithin{equation}{section}

\begin{document}

\newcommand{\pd}{\partial}
\newcommand{\res}{\mathrm{res}}
\newcommand{\alg}{\mathrm{alg}}
\newcommand{\Gr}{\mathrm{Gr}}
\newcommand{\Grad}{\mathrm{Gr^{ad}}}
\newcommand{\Ai}{\mathrm{Ai}}
\newcommand{\Span}{\mathrm{span}}
\newcommand{\clspan}{\overline{\mathrm{span}}}

\newcommand{\Rset}{\mathbb{R}}
\newcommand{\Cset}{\mathbb{C}}
\newcommand{\Nset}{\mathbb{N}}
\newcommand{\Zset}{\mathbb{Z}}
\newcommand{\cA}{\mathfrak{A}}
\newcommand{\cC}{\mathcal{C}}
\newcommand{\cF}{\mathcal{F}}
\newcommand{\cH}{\mathcal{H}}
\newcommand{\cK}{\mathcal{K}}
\newcommand{\cL}{\mathcal{L}}
\newcommand{\cM}{\mathcal{M}}
\newcommand{\cP}{\mathcal{P}}
\newcommand{\cQ}{\mathcal{Q}}
\newcommand{\cR}{\mathcal{R}}
\newcommand{\cS}{\mathcal{S}}
\newcommand{\cU}{\mathcal{U}}
\newcommand{\cV}{\mathcal{V}}

\newcommand{\sH}{\mathsf{H}}

\newcommand{\fR}{\mathfrak{R}}

\newcommand{\al}{\alpha}
\newcommand{\be}{\beta}
\newcommand{\ka}{\kappa}
\newcommand{\la}{\lambda}
\newcommand{\om}{\omega}

\newcommand{\tS}{\tilde{S}}

\title{Higher-order heat equation and the Gelfand-Dickey hierarchy}

\author[P.~Iliev]{Plamen~Iliev}
\address{School of Mathematics, Georgia Institute of Technology, 
Atlanta, GA 30332--0160, USA}
\email{iliev@math.gatech.edu}
\thanks{The author gratefully acknowledges the support of a Simons Foundation Grant \#635462 and a CRM-Simons Professorship at the Centre de Recherches Math\'ematiques, Universit\'e de Montr\'eal.}



\begin{abstract} 
In this paper we analyze the heat kernel of the equation $\partial_tv =\pm\mathcal{L} v$, where $\mathcal{L}=\partial_x^N+u_{N-2}(x)\partial_x^{N-2}+\cdots+u_0(x)$ is an $N$-th order differential operator and the $\pm$ sign on the right-hand side is chosen appropriately. Using formal pseudo-differential operators, we derive an explicit formula for Hadamard's coefficients in the expansion of the heat kernel in terms of the resolvent of $\mathcal{L}$. Combining this formula with soliton techniques and Sato's Grassmannian, we establish different properties of Hadamard's coefficients and relate them to the Gelfand-Dickey hierarchy. In particular, using the correspondence between commutative rings of differential operators and algebraic curves due to Burchnall-Chaundy and Krichever, we prove that the heat kernel consists of finitely many terms if and only if the operator $\mathcal{L}$ belongs to a rank-one commutative ring of differential operators whose spectral curve is rational with only one cusp-like singular point, and the coefficients $u_j(x)$ vanish at $x=\infty$. We also characterize these operators $\mathcal{L}$ as the rational solutions of the Gelfand-Dickey hierarchy with coefficients $u_j $ vanishing at $x=\infty$, or as the rank-one solutions of the bispectral problem vanishing at $\infty$. 
\end{abstract}

\maketitle

\tableofcontents

\section{Introduction} \label{se1}

The heat kernel of self-adjoint second-order operators plays an important role in spectral theory and has numerous applications in geometry and theoretical physics. During the last few years, there have been numerous probabilistic constructions analyzing the higher-order heat equation
\begin{equation}\label{1.1}
\pd_tv =\pm\pd^{N}_xv, 
\end{equation}
and associated stochastic processes, see for instance \cite{ACM,GHPD} and the references therein. Moreover, the Airy-type integrals representing the fundamental solution of the heat equation \eqref{1.1} naturally appear in the theory of two-dimensional quantum gravity and random matrices \cite{AvM,DGZ,Kontsevich} which has seen remarkable developments in the recent years. 

The more general equation 
\begin{equation}\label{1.2}
\pd_tv =\ka_N\cL v, 
\end{equation}
where $\ka_N=\pm 1$ and $\cL$ is a differential operator of arbitrary order $N$ has also appeared in different applications, from the theory of thermal grooving \cite{Mull} to trimolecular reactions \cite{Gar}. The construction of a functional integral representation of the solution of PDEs of the form in \eqref{1.2} based on the theory of infinite dimensional Fresnel integrals was obtained recently in \cite{ACM}.

In the present paper, we take a different path and study the heat kernel of equation \eqref{1.2} 
for a general $N$-th order differential operator of the form
$$\cL=\pd_x^N+u_{N-2}(x)\pd_x^{N-2}+\cdots+u_0(x)$$ 
and $\ka_N=\pm1$ by exploring connections to integrable systems.   
Recall that the fundamental solution of $\pd_tv =(\pd_x^2+u_0(x)) v$ has an asymptotic expansion of the form 
\begin{equation}\label{heat-kernel-exp}
v(x,y,t)\sim\frac{e^{-\frac{(x-y)^2}{4t}}}{\sqrt{4\pi t}}\left(1+\sum_{k=1}^{\infty}H_k(x,y)t^k\right), \qquad \text{ as }\qquad t\to 0+
\end{equation}
see \cite{MP}. The expansion on the right-hand side is characterized by the following two conditions: 
\begin{enumerate}
\item for fixed $y$, it satisfies the heat equation (as formal power series),  and 
\item the Hadamard's coefficients $H_k(x,y)$ are smooth functions in some neighborhood of the diagonal $x=y$.
\end{enumerate}
This can  be generalized in a natural way for arbitrary $N>2$ by replacing the Gaussian $\frac{1}{\sqrt{4\pi t}}e^{-\frac{(x-y)^2}{4t}}$ in \eqref{heat-kernel-exp} when $N=2$ with an appropriate Airy-type function determined from the fundamental solution of \eqref{1.1}. Using formal pseudo-differential operators, we prove that Hadamard's coefficients in the corresponding expansion of the heat kernel can be computed from the resolvent of the operator $\cL$. This provides an extension and a new derivation of a formula discovered in \cite{I05} for the Hadamard's coefficients in \eqref{heat-kernel-exp} when $\cL=\pd_x^2+u_0(x)$ is a second-order differential operator. We use this to establish different properties of Hadamard's coefficients when $N>2$ by applying soliton techniques and exploring connections with Sato's Grassmannian. For instance, the link between the values of the resolvent on the diagonal and the equations of the Gelfand-Dickey hierarchy \cite{GD} shows that we can define first integrals of the Gelfand-Dickey hierarchy in terms of the heat kernel, extending the classical theory for Schr\"odinger operators and the Korteweg-de Vries (KdV) hierarchy \cite{McKvM,Sch}. However, already when $N=3$, the values of Hadamard's coefficients on the diagonal are no longer sufficient to generate the equations in the Boussinesq hierarchy, and thus the explicit formula outside the diagonal $x=y$ can be used to determine the appropriate higher-order terms needed to define the flows. 

As another interesting application we characterize the operators $\cL$ for which the heat kernel of \eqref{1.2} consists of finitely many terms. Note that if the heat kernel  has finitely many terms, it provides an exact formula for the fundamental solution which is valid for all $t>0$. When $N=2$, the corresponding Schr\"odinger operators $\cL=\pd_x^2+u_0$ appear in many different settings and can be described in several ways: as the rational solutions of the KdV equation vanishing at $x=\infty$,  as rational Darboux transformations from the operator $\pd_x^2$, as the rank-one solutions of the bispectral problem, or within the context of Huygens’ principle \cite{AM,AMM, Be,BCE,CFV,CV,DG}. In particular, the pioneering work of Duistermaat and Gr\"unbaum \cite{DG} on the bispectral problem inspired numerous far-reaching extensions of these beautiful connections for higher-order differential operators \cite{BHY,Gr1,Gr2,Wilson}, as well as new instances of finite heat kernel expansions for discrete second-order operators \cite{GI,Ha1,I08}. The higher-order operators $\cL$  for which the heat kernel of \eqref{1.2} has finitely many terms can be naturally linked to these works, which allows us to interpret the finiteness of the heat kernel within the context of algebraic geometry, the theory of soliton equations, or bispectrality. More precisely, using the Burchnall-Chaundy-Krichever correspondence \cite{BC,Krichever,Mumford} between commutative rings of differential operators and algebraic curves, we prove that the heat kernel consists of finitely many terms if and only if the operator $\cL$ belongs to a rank-one commutative ring of differential operators whose spectral curve  is rational with only one cusp-like singular point, and the coefficients $u_j (x)$ vanish at $\infty$.  Combining this with the work of Krichever \cite{Krichever2} on the rational solutions of the Kadomtsev-Petviashvili (KP) equation and the work of Wilson \cite{Wilson} on the classification of rank-one bispectral commutative rings of differential operators, we show that the operators $\cL$ for which the heat expansion is finite can be characterized as the rational solutions of the Gelfand-Dickey hierarchy with coefficients $u_j $ vanishing at $x=\infty$, or as the rank-one solutions of the bispectral problem vanishing at $\infty$. 

The paper is organized as follows. In the next section, we introduce the Airy-type functions which appear in the fundamental solution of the heat equation \eqref{1.1}. In \seref{se3}, we define the heat kernel of \eqref{1.2} by extending the classical approach \cite{Hadamard,MP}, and we illustrate the recursive definition of Hadamard's coefficients when $N=3$. In \seref{se4}, we establish the connection between the heat kernel and the resolvent in terms of formal pseudo-differential operators, deduce the formula for Hadamard's coefficients and discuss some of their properties. We also explain how this approach leads to the formula for Hadamard's coefficients discovered in \cite{I05} when $N=2$. In \seref{se5}, we show that the values of Hadamard's coefficients on the diagonal provide first integrals for the Gelfand-Dickey hierarchy, and that by adding the first derivatives on the diagonal we can generate the equations of the Boussinesq hierarchy. In \seref{se6}, we prove that the heat kernel consists of finitely many terms if and only if the operator $\cL$ belongs to a rank-one commutative ring of differential operators whose spectral curve  is rational with only one cusp-like singular point, and the coefficients $u_j (x)$ vanish at $\infty$. This shows that all such operators are parametrized by the sub-Grassmannian $\Gr_0^{(N)}$ of Sato-Segal-Wilson Grassmannian. In \seref{se7}, we provide several explicit examples of third and higher-order operators $\cL$ having finite heat kernel expansions. In \seref{se8}, we explain how the operators $\cL$ for which the heat kernel is finite can be characterized as the rational solutions of the Gelfand-Dickey hierarchy vanishing at $\infty$, and in \seref{se9}, we briefly discuss the connection to the bispectral problem and the Darboux transformation.

\section{The fundamental solution of the heat equation $\pd_t v=\pm\pd^{N}_xv $}  \label{se2}

In this section, we discuss the fundamental solution of the heat equation
\begin{equation}\label{2.1}
\pd_tv =\ka_N\pd^{N}_xv, \quad x\in\Rset, \quad t>0,
\end{equation}
where $\ka_N=(-1)^{N/2+1}$ if $N$ is even, and $\ka_N=\pm1 $ if $N$ is odd, i.e.
\begin{equation}\label{2.2}
\ka_N=\begin{cases}
(-1)^{N/2+1} & \text {\text if $N$ is even,}\\
\pm 1 &\text{if $N$ is odd.}
\end{cases}
\end{equation}
The case when $N$ is even was studied as early as 1960 in \cite{Krylov}, and a very detailed analysis of the initial-value problem can be found in \cite{DM}. However, the integrals representing the fundamental solution have been investigated even before that in the works of Bernstein \cite{Bernstein}, Burwell \cite{Burwell} and P\'olya \cite{Polya}. For a detailed analysis of the fundamental solution and its representation in terms of hypergeometric functions, see \cite{GHPD} and the references therein. 

Below, we review briefly the construction of the fundamental solution of \eqref{2.1} and we fix the notations that will be used throughout the paper. For fixed $N$ and $\ka_N$,  we first consider the solution of  \eqref{2.1} with initial condition 
$$v|_{t=0}=\delta(x),$$ 
where $\delta$ is the Dirac delta function. By applying the Fourier transform, we can represent the solution as the integral 
$$\frac{1}{2\pi}\int_{-\infty}^{\infty}e^{ix\xi+\ka_N(i\xi)^Nt}d\xi.$$
We set $t=1$ in the last formula and we denote the corresponding Airy-type integral by $A_N(x;\ka_N)$, i.e.
\begin{equation}\label{2.3}
A_N(z;\ka_N)=\frac{1}{2\pi}\int_{-\infty}^{\infty}e^{iz\xi+\ka_N(i\xi)^N}d\xi.
\end{equation}
When $N$ is even, the integral converges absolutely and defines an entire function of $z$. When $N$ is odd, the integral does not converge absolutely, but this can be fixed by moving the integration in the complex plane over the line $c+i\Rset$, where $c$ is chosen appropriately depending on the sign of $\ka_N$. Indeed, if $N$ is odd and $\ka_N=(-1)^{(N+1)/2}$   then for fixed $z$, we take $c>0$ and we apply Goursat's theorem by integrating the entire function $e^{z\xi+\ka_N\xi^N}$ over the rectangle with vertices $\xi=\pm iR, c\pm iR$. It is easy to see that the integrals over the horizontal line segments $u\pm iR$, where $u\in[0,c]$, approach $0$ as $R\to\infty$ and therefore 
\begin{equation}\label{2.4}
A_N(z;\ka_N)=\frac{1}{2\pi i}\int_{c-i\infty}^{c+i\infty}e^{z\xi+\ka_N\xi^N}d\xi,
\end{equation}
where $c>0$ is arbitrary. Similarly, if $N$ is odd and $\ka_N=(-1)^{(N-1)/2}$, equation \eqref{2.4} holds for $c<0$. The integral in \eqref{2.4} converges absolutely and defines an entire function of $z$ when $N$ is odd. Differentiation under the integral shows that the function $A_N(z;\ka_N)$ satisfies the higher-order Airy equation
\begin{equation}\label{2.5}
A_N^{(N-1)}(z;\ka_N)=-\frac{\ka_N}{N}zA_N(z;\ka_N).
\end{equation}
With the notations above, the fundamental solution of equation \eqref{2.1}, which satisfies the initial condition
$$v|_{t=0}=\delta(x-y)$$ 
can be written as follows
\begin{equation}\label{2.6}
v(x,y,t)=\frac{1}{\sqrt[N]{t}}A_N\left(\frac{x-y}{\sqrt[N]{t}};\ka_N\right).
\end{equation}

\begin{Remark}
If $N$ is even, the sign of $\ka_N$ is fixed and we have a unique function $A_N(z;\ka_N)$.  If $N$ is odd, we have two functions which are related by
\begin{equation}\label{2.7}
A_N(z;-\ka_N)=A_N(-z;\ka_N).
\end{equation}
In order to simplify the notation, throughout the paper, we omit the explicit dependence of $\ka_N$, and we will write simply $A_N(z)$, unless $N$ is odd and the dependence on $\ka_N$ is important.
\end{Remark}

We can rewrite $A_N(z)$ in terms of real-valued functions as follows. 
\begin{enumerate}[(i)]
\begin{subequations}
\item When $N$ is even, we have 
\begin{align}
A_N(z)= 
\frac{1}{2\pi}\int_{-\infty}^{\infty}e^{isz-s^N}ds 
=\frac{1}{\pi}\int_0^{\infty}\cos (zs)e^{-s^N}\,ds.\label{2.8a}
\end{align}
\item When $N$ is odd, we have 
\begin{align}
A_N(z;\ka_N)&=\frac{1}{2\pi}\int_{-\infty}^{\infty}e^{isz+i\ka_N(-1)^{(N-1)/2}s^N}ds  \nonumber \\
&=\begin{cases}
\frac{1}{\pi}\int_0^{\infty}\cos (zs+s^N)\,ds & \text {\text if $\ka_N=(-1)^{(N-1)/2}$,}\\
\frac{1}{\pi}\int_0^{\infty}\cos (zs-s^N)\,ds & \text {\text if $\ka_N=(-1)^{(N+1)/2}$.}\label{2.8b}
\end{cases}
\end{align}
\end{subequations}
\end{enumerate}

\begin{Example}\label{Ex2.2}
If $N=2$, equation \eqref{2.8a} gives
\begin{equation}\label{2.9}
A_2(x)=\frac{1}{\pi}\int_0^{\infty}\cos (xs)e^{-s^2}\,ds=\frac{1}{\sqrt{4\pi}}e^{-x^2/4},
\end{equation}
and therefore \eqref{2.6} yields the well-known formula for the fundamental solution of the heat equation on the real line
$$v(x,y,t)=\frac{1}{\sqrt{4\pi t}}e^{-\frac{(x-y)^2}{4t}}.$$
\end{Example}

\begin{Example}\label{Ex2.4}
If $N=3$ and $\ka_3=1$, equation \eqref{2.8b} gives
\begin{equation}\label{2.10}
A_3(x;1)=\frac{1}{\pi}\int_0^{\infty}\cos (xs-s^3)\,ds = \frac{1}{\sqrt[3]{3}}\Ai\left(-\frac{x}{\sqrt[3]{3}}\right),
\end{equation}
where $\Ai(x)$ is the Airy function \cite[p.~447, equation 10.4.32]{AS}.
The formula for the  fundamental solution in this case becomes
$$v(x,y,t)=\frac{1}{\sqrt[3]{3t}}\Ai\left(-\frac{x-y}{{\sqrt[3]{3t}}}\right).$$
\end{Example}

\section{Heat kernel for a general $N$-th order operator} \label{se3}
We fix an $N$-th order differential operator $\cL$ of the form
\begin{equation}\label{3.1}
\cL=\pd_x^N+u_{N-2}(x)\pd_x^{N-2}+\cdots+u_0(x), \qquad\text{ where }\pd_x=\frac{d}{dx},
\end{equation}
i.e.  we will assume that $\cL$ is monic and that the coefficient of $\pd_x^{N-1}$ is $0$, which can always be achieved by an appropriate change of variable and a gauge transformation. 
For such operators $\cL$ we  consider the heat equation
\begin{equation}\label{3.2}
\pd_tv =\ka_N\cL v.
\end{equation}
We define the {\em  heat kernel}  of  \eqref{3.2} as the formal power series expansion
\begin{align}
\cH(x,y,t)=&\frac{1}{\sqrt[N]{t}}A_N\left(\frac{x-y}{\sqrt[N]{t}}\right)
+\sum_{j=0}^{N-2}\frac{1}{\sqrt[N]{t^{j+1}}}A_N^{(j)}\left(\frac{x-y}{\sqrt[N]{t}}\right)\left(\sum_{k=1}^{\infty}H_k^{j}(x,y)t^k\right)\nonumber\\
=&\frac{1}{\sqrt[N]{t}}A_N\left(\frac{x-y}{\sqrt[N]{t}}\right)\left(1+\sum_{k=1}^{\infty}H_k^{0}(x,y)t^k\right)\nonumber\\
&+\frac{1}{\sqrt[N]{t^2}}A_N'\left(\frac{x-y}{\sqrt[N]{t}}\right)\left(\sum_{k=1}^{\infty}H_k^{1}(x,y)t^k\right)\nonumber\\
&+\cdots+\frac{1}{\sqrt[N]{t^{N-1}}}A_N^{(N-2)}\left(\frac{x-y}{\sqrt[N]{t}}\right)\left(\sum_{k=1}^{\infty}H_k^{N-2}(x,y)t^k\right),\label{3.3}
\end{align}
where $t>0$ and the coefficients $H_k^{j}(x,y)$  are such that:
\begin{enumerate}[(i)]
\item For fixed $y$, $v(x,t)=\cH(x,y,t)$ is a solution of the heat equation \eqref{3.2};
\item The coefficients $H_k^{j}(x,y)$ are smooth functions in a neighborhood of the diagonal $y=x$.
\end{enumerate}
It is not hard to see that these two properties fix $\cH(x,y,t)$ uniquely. Indeed, if we substitute \eqref{3.3} into \eqref{3.2} and if we compare the coefficients of 
$\frac{1}{t^{(j+1)/N}}A_N^{(j)}\left(\frac{x-y}{\sqrt[N]{t}}\right)t^{k}$, where $j\in\{0,\dots,N-2\}$ and $k\in\Nset_0$, 
we obtain an equation of the form
\begin{equation}\label{3.4}
(k+1)H_{k+1}^{j}(x,y)+(x-y) \frac{\pd H_{k+1}^{j}(x,y)}{\pd x}=\cF_{kj}(H_\ell^{i}(x,y)),
\end{equation}
where the right-hand side is a function $\cF_{kj}$ of the coefficients $H_\ell^{i}(x,y)$ and their derivatives with respect to $x$, where $\ell=k+1$ and $i>j$, or $\ell=k$. Since this equation has a unique smooth solution in some neighborhood of the diagonal $y=x$, we can determine the coefficients $H_k^{j}(x,y)$ recursively as follows: 
$$H_1^{N-2}(x,y), H_1^{N-3}(x,y), \dots, H_1^{0}(x,y), H_2^{N-2}(x,y),\dots, H_2^{0}(x,y),H_3^{N-2}(x,y),\dots.$$
Similarly to the case of second-order operators \cite{MP}, we see that the truncated series in \eqref{3.3} approximate the fundamental solution of \eqref{3.2}. Therefore, we can think of the series in \eqref{3.3} as an appropriate analog of the asymptotic expansion as $t\to0+$ of the fundamental solution of the heat equation \eqref{3.2}. In particular, we classify in \seref{se6} the operators $\cL$ for which the expansion  in \eqref{3.3} has finitely many terms, so the series gives a closed formula for the fundamental solution; see \seref{se7} for several explicit examples. Following the convention in the literature in the classical case $N=2$, we will refer to the coefficients $H_k^{j}(x,y)$ as the {\em Hadamard's coefficients} for the heat equation \eqref{3.2}. 

\begin{Remark}
For every $j$ and $k$, the coefficients in the expansion of $H_k^{j}(x,y)$ about the diagonal $y=x$  are differential polynomials of the coefficients $u_0(x),\dots,u_{N-2}(x)$ of the operator $\cL$. In particular, when $N=2$, the values of Hadamard's coefficients $H_k^{0}(x,x)$ on the diagonal  provide first integrals of the KdV hierarchy and can be used to obtain all equations in the hierarchy, see \reref{re5.2}. When $N>2$, $H_k^{j}(x,x)$ are still first integrals for the $N$-th Gelfand-Dickey hierarchy as shown in \thref{th5.1}. However, the connection between the equations of the $N$-th Gelfand-Dickey hierarchy and Hadamard's coefficients is more subtle when $N>2$, see Proposition~\ref{pr5.3} for explicit formulas for the equations of the Boussinesq hierarchy in terms of the heat kernel when $N=3$.
\end{Remark}

\begin{Example}[Heat kernel for second-order operators]\label{Ex3.2}
If $N=2$, the operator in \eqref{3.1} is the Schr\"odinger operator 
$$\cL=\pd_x^2+u_0(x).$$
The function $A_2(x)$ is the Gaussian function in \eqref{2.9} and the heat kernel expansion in \eqref{3.3} becomes
\begin{equation}\label{3.5}
\cH(x,y,t)=\frac{e^{-\frac{(x-y)^2}{4t}}}{\sqrt{4\pi t}}\left(1+\sum_{k=1}^{\infty}H_k^{0}(x,y)t^k\right).
\end{equation}
The coefficients $H_k^{0}(x,y)$ in the above expansion are the classical Hadamard's coefficients which are defined recursively via the differential equations
\begin{equation}\label{3.6}
(k+1)H_{k+1}^{0}(x,y)+(x-y) \frac{\pd H_{k+1}^{0}(x,y)}{\pd x}=\cL H_{k}^{0}(x,y), \text{ for }k\in\Nset_0,
\end{equation}
where $\cL H_{k}^{0}(x,y)= \frac{\pd^2 H_{k}^{0}(x,y)}{\pd x^2}+u_0(x)H_{k}^{0}(x,y)$, and $H_{0}^{0}(x,y)=1$.
From this relation, one can compute  consecutively the coefficients in the expansion of $H_{k}^{0}(x,y)$ about the diagonal $y=x$ and express them as differential  polynomials of the potential $u_0(x)$. For instance, 
\begin{align*}
\frac{\pd^j H_{1}^{0}(x,y)}{\pd x^j}\Big|_{y=x}&=\frac{1}{j+1} u_0^{(j)}(x),\\ 
\frac{\pd^j H_{2}^{0}(x,y)}{\pd x^j}\Big|_{y=x}&=\frac{1}{(j+2)(j+3)} u_0^{(j+2)}(x)+\frac{1}{j+2}\sum_{\ell=0}^{j}\frac{1}{\ell+1}\binom{j}{\ell}u_0^{(\ell)}(x)u_0^{(j-\ell)}(x),
\end{align*} 
see \cite[Theorem 5.3]{Sch} for the general formula. Moreover, 
$$H_{1}^{0}(x,x)=u_0(x),\quad  H_{2}^{0}(x,x)=\frac{1}{6}(u_0''(x)+3u_0^2(x)),\dots$$ 
are first integrals of the KdV hierarchy.
\end{Example}

\begin{Example}[Heat kernel for third-order operators]\label{Ex3.3}
If $N=3$ and $\ka_3=1$, then $A_3(z;1)$  is the Airy function in \eqref{2.10}. For the third-order operator 
$$\cL=\pd_x^3+u_1(x)\pd_x+u_0(x),$$
the heat kernel of equation \eqref{3.2} will have the expansion
\begin{align}
\cH(x,y,t)=& \frac{1}{\sqrt[3]{3t}}\Ai\left(-\frac{x-y}{\sqrt[3]{3t}}\right)\left(1+\sum_{k=1}^{\infty}H_k^{0}(x,y)t^k\right)\nonumber\\
&\qquad-  \frac{1}{\sqrt[3]{9t^2}}\Ai'\left(-\frac{x-y}{\sqrt[3]{3t}}\right) \left(\sum_{k=1}^{\infty}H_k^{1}(x,y)t^k\right). \label{3.7}
\end{align}
The coefficients $H_k^{0}(x,y)$, $H_k^{1}(x,y)$ satisfy the differential equations
\begin{subequations}
\begin{align}
&(k+1)H_{k+1}^{1}(x,y)+(x-y) \frac{\pd H_{k+1}^{1}(x,y)}{\pd x}\nonumber\\ 
&\qquad=\cL H_{k}^{1}(x,y) + 3\frac{\pd^2 H_k^{0}(x,y)}{\pd x^2}+u_1(x) H_k^{0}(x,y),\label{3.8a}\\
&(k+1)H_{k+1}^{0}(x,y)+(x-y) \frac{\pd H_{k+1}^{0}(x,y)}{\pd x} \nonumber\\
&\qquad=\cL H_{k}^{0}(x,y) - \frac{(x-y)\pd^2 H_{k+1}^{1}(x,y)}{\pd x^2}-\frac{\pd H_{k+1}^{1}(x,y)}{\pd x}-\frac{(x-y)u_1(x)H_{k+1}^{1}(x,y)}{3},\label{3.8b}
\end{align}
\end{subequations}
where $H_{0}^{1}(x,y)=0$ and $H_{0}^{0}(x,y)=1$. Using the above equations, one can show directly by induction that the coefficients in the expansion of $H_{k}^{j}(x,y)$ about the diagonal $y=x$ are differential polynomials of $u_0(x)$ and $u_1(x)$. For instance, from \eqref{3.8a} with $k=0$ it is easy to see that 
$$H_{1}^{1}(x,x)=u_1(x),\quad \text{ and more generally }\quad\frac{\pd^j H_{1}^{1}(x,y)}{\pd x^j}\Big|_{y=x}=\frac{1}{j+1} u_1^{(j)}(x).$$ 
Using now \eqref{3.8b} with $k=0$ one can deduce that 
$$H_{1}^{0}(x,x)=u_0(x)-\frac{u_1'(x)}{2}, \qquad \frac{\pd H_{1}^{0}(x,y)}{\pd x}\Big|_{y=x}=\frac{1}{2} u_0'(x)-\frac{1}{3}u_1''(x) -\frac{1}{6}u_1^2(x), $$
and so on.
\end{Example}

\begin{Remark}\label{re3.4}
When $N$ is odd, we can easily connect the heat kernel of the equation $\pd_tv =\ka_N\cL v$ to the heat kernel of $\pd_tv =-\ka_N\cL v$, by formally replacing $t$ with $-t$. More precisely, if $\cH(x,y,t;\ka_N)$ and $\cH(x,y,t;-\ka_N)$  denote the heat kernels of the first and the second equation, respectively, then from \eqref{2.7} it follows that 
$$\cH(x,y,t;\ka_N)=-\cH(x,y,-t;-\ka_N).$$ 
For the coefficients $H_k^{j}$, where $k\in\Nset$ and $j\in\{0,\dots,N-2\}$, this means that 
$$H_k^{j}(x,y;\ka_N)=(-1)^{k}H_k^{j}(x,y;-\ka_N).$$ 
\end{Remark}

\section{Hadamard's coefficients and the resolvent} \label{se4}
We start by reviewing some of the basic properties of formal pseudo-differential operators, using the standard notations in the literature \cite{Dickey,PvM}. 
A formal pseudo-differential operator is a formal series of the form 
\begin{equation}\label{4.1}
R=\sum_{i=-\infty}^mr_i(x)\pd_x^i,
\end{equation}
where $m\in\Zset$ and $r_i(x)$ are functions of $x$. The product of two operators is defined by the following generalization of the Leibniz rule:
\begin{equation*}
\pd_x^k\cdot r(x)=\sum_{i=0}^{\infty}\binom{k}{i}\frac{d^ir(x)}{dx^i}\,\pd_x^{k-i}, \qquad\text{ where }k\in\Zset.
\end{equation*}
It is straightforward to check that, with the product defined above, the set of pseudo-differential operators becomes an associative ring. An important feature in this ring is that we can define an $m$-th root and an inverse for any pseudo-differential operator $R$ of the form \eqref{4.1} with leading coefficient $r_m=1$. Indeed, one can show that there exist unique 
pseudo-differential operators $R^{1/m}$ and $R^{-1}$ of the form
\begin{align*}
R^{1/m}=\pd_x+\sum_{i\leq0}p_i(x)\pd_x^i, \qquad R^{-1}=\pd_x^{-m}+\sum_{j<-m}q_j(x)\pd_x^j,
\end{align*}
such that
\begin{equation*}
R=\left(R^{1/m}\right)^m \qquad {\mathrm{and}}\qquad RR^{-1}=1.
\end{equation*}
Moreover, the operators $R$, $R^{-1}$, and $R^{1/m}$ commute with each other. For a formal pseudo-differential operator $R$ we write $R_+$ for 
the {\it differential operator part}
\begin{equation*}
R_+=\sum_{i\geq 0}r_i(x)\pd_x^i,
\end{equation*}
and $R_-$ for the {\it integral (or Volterra) part\/}
\begin{equation*}
R_-=\sum_{i<0}r_i(x)\pd_x^i.
\end{equation*}
Within the ring of formal pseudo-differential operators we can  represent the operator $\cL$ in \eqref{3.1} in the dressing form 
\begin{equation}\label{4.2}
\cL=S\pd_x^N S^{-1},
\end{equation}
where $S=S(x,\pd_x)$ is a formal pseudo-differential operator of the form
\begin{equation}\label{4.3}
S=\sum_{k=0}^{\infty}\psi_k(x)\pd_x^{-k}, \quad \psi_0(x)=1.
\end{equation}
The operator $S$ is determined uniquely from $\cL$, up to right multiplication by a constant coefficient  operator $S_0=1+\sum_{j=1}^{\infty}c_j\pd_x^{-j}$. 

The adjoint antiautomorphism on the ring of formal pseudo-differential operators is defined by setting $\pd_x^*=-\pd_x$, and we denote by $R^*$ the adjoint of an operator $R$. Explicitly, if $R$ is the operator in \eqref{4.1}, then 
$$R^*=\sum_{i=-\infty}^m (-\pd_x)^i\cdot r_i(x).$$
Since $\pd_xe^{xz}=ze^{xz}$, we can define $\pd_x^{j}e^{xz}=z^{j}e^{xz}$ for all $j\in\Zset$. The stationary wave function $\Psi(x,z)$ and the adjoint wave function $\Psi^*(x,z)$ are defined by applying $S$ and $(S^{*})^{-1}$ to $e^{xz}$ and $e^{-xz}$, respectively:
\begin{subequations}\label{4.4}
\begin{align}
\Psi(x,z)&= S e^{xz}=\left(\sum_{k=0}^{\infty}\psi_k(x) z^{-k}\right)e^{xz}\label{4.4a}\\
\Psi^*(x,z)&= (S^*)^{-1} e^{-xz} =\left(\sum_{k=0}^{\infty}\psi^*_k(x) z^{-k}\right)e^{-xz}.\label{4.4b}
\end{align}
\end{subequations}
From equations \eqref{4.2} and \eqref{4.4} it follows that 
\begin{equation}\label{4.5}
\cL\Psi(x,z)=z^N\Psi(x,z) \quad\text{ and }\quad\cL^*\Psi^*(x,z)=z^N\Psi^*(x,z).
\end{equation}
We denote by  $\fR(\cL;x,y,z)$, or simply by $\fR(\cL)$,  the product of $\Psi(x,z)\Psi^*(y,z)$, i.e.
\begin{equation}\label{4.6}
\fR(\cL)=\Psi(x,z)\Psi^*(y,z)=\left(\sum_{n=0}^{\infty}\frac{\om_n(x,y)}{z^n}\right)e^{(x-y)z}.
\end{equation}
\begin{Remark}
Note that while the operator $S$ is determined up to a multiplication by a formal pseudo-differential operator $S_0=1+\sum_{j=1}^{\infty}c_j\pd_x^{-j}$ with constant coefficients on the right, the product $\Psi(x,z)\Psi^*(y,z)$ is uniquely determined by $\cL$, and thus the coefficients $\om_n(x,y)$ are well-defined and depend only on $\cL$. 
From \eqref{4.5} it follows that $\fR(\cL)$ satisfies the equations
\begin{align*}
&(\cL -z^N)\,\fR(\cL) =0,\\
&(\cL_y^{*}-z^N) \,\fR(\cL)= 0,
\end{align*}
where $\cL_y^{*} $ means that the operator $\cL^{*}$ acts on the variable $y$. Recall that an integral operator whose kernel satisfies the above equations is called the {\em resolvent} of the operator $\cL$, cf. \cite[Section 3]{GD} or \cite[Section 1.8.5]{Dickey}. In most applications, the resolvent is considered only on the diagonal $y=x$, which allows to write everything as differential polynomials of the coefficients $u_0(x),\dots,u_{N-2}(x)$ of the operator $\cL$. However, in order to describe the properties of Hadamard's coefficients, we will need also the non-diagonal values of $\om_n(x,y)$.
\end{Remark}

Now we want to associate a formal pseudo-differential operator with the heat kernel \eqref{3.3} and to connect it to the above constructions.
Since $A_N(z)$  solves the Airy equation \eqref{2.5}, it follows that the fundamental solution $v(x,y,t)=\frac{1}{\sqrt[N]{t}}A_N\left(\frac{x-y}{\sqrt[N]{t}}\right)$ of the heat equation \eqref{2.1}, satisfies the equation
\begin{equation}\label{4.7}
t \pd_x^{N-1} \frac{1}{\sqrt[N]{t}}A_N\left(\frac{x-y}{\sqrt[N]{t}}\right) = -\frac{\ka_N(x-y)}{N} \frac{1}{\sqrt[N]{t}}A_N\left(\frac{x-y}{\sqrt[N]{t}}\right).
\end{equation}
This suggests that we can interpret  $tv(x,y,t)$ as $ -\frac{\ka_N}{N}\pd_x^{-N+1} (x-y)\cdot v(x,y,t)$ and rewrite the heat kernel $\cH(x,y,t)$ in \eqref{3.3} as
\begin{equation}\label{4.8}
\cH(x,y,t)= H (x,y,\pd_x)\cdot \frac{1}{\sqrt[N]{t}}A_N\left(\frac{x-y}{\sqrt[N]{t}}\right),
\end{equation}
where
\begin{equation}\label{4.9}
H (x,y,\pd_x)= 1+ \sum_{j=0}^{N-2}\sum_{k=1}^{\infty}H_k^{j}(x,y)\frac{\ka_N^k}{N^k} \pd_x^{j} D_N^k,
\end{equation}
and 
\begin{equation}\label{4.10}
D_N= -\pd_x^{-N+1} (x-y) = -(x-y)\pd_x^{-N+1}  +(N-1)\pd_x^{-N}. 
\end{equation}

With the above notations, we can establish the following relation between the heat kernel and the resolvent.

\begin{Theorem}\label{th4.2}
If we expand the operator $H (x,y,\pd_x)$ in \eqref{4.9} in powers of $\pd_x^{-1}$ we have 
\begin{equation}\label{4.11}
H (x,y,\pd_x)=\sum_{n=0}^{\infty}\om_n(x,y)\pd_x^{-n}.
\end{equation}
\end{Theorem}

\begin{Remark}\label{re4.3}
On the left-hand side of \eqref{4.11}, which is defined in \eqref{4.9}, it is important to treat $x$ and $y$ as independent variables until the powers of $\pd_x^{-1}$ are pulled to the right. In particular, this is how we interpret this relation on the diagonal $y=x$.
\end{Remark}

\begin{proof}
Let 
\begin{equation}\label{4.12}
R=\sum_{n=0}^{\infty}\om_n(x,y)\pd_x^{-n}.
\end{equation}
If we set $\tilde{\cH}(x,y,t)=R\, \frac{1}{\sqrt[N]{t}}A_N\left(\frac{x-y}{\sqrt[N]{t}}\right)$, then the proof will follow if we can show that:
\begin{enumerate}[(i)]
\item For fixed $y$, the function  $\tilde{\cH}(x,y,t)$ satisfies the heat equation \eqref{3.2};
\item In the expansion \eqref{3.2}, the coefficients $H_k^{j}(x,y)$ are smooth functions in some neighborhood of the diagonal $y=x$.
\end{enumerate}
From  \eqref{4.4} and \eqref{4.6} it follows that $R=S_x\tS_y$, where $S_x=\sum_{k=0}^{\infty}\psi_k(x)\pd_x^{-k}$ is the wave operator in \eqref{4.3}, and 
$\tS_y=\sum_{k=0}^{\infty}\psi_k^{*}(y)\pd_x^{-k}$ is an operator with coefficients which are independent of $x$. Therefore,
\begin{align*}
\frac{\pd \tilde{\cH}(x,y,t)}{\pd t}&=S_x\tS_y \frac{\pd }{\pd t} \left(\frac{1}{\sqrt[N]{t}}A_N\left(\frac{x-y}{\sqrt[N]{t}}\right)\right) 
=\ka_NS_x\tS_y \,\pd_x^N\,\left(\frac{1}{\sqrt[N]{t}}A_N\left(\frac{x-y}{\sqrt[N]{t}}\right)\right)\\
&=\ka_N(S_x \pd_x^{N} S_x^{-1}) R\,\frac{1}{\sqrt[N]{t}}A_N\left(\frac{x-y}{\sqrt[N]{t}}\right)=\ka_N\cL \tilde{\cH}(x,y,t),
\end{align*}
proving (i). Now, we  look at the limit $y\to x$ and we want to show that \eqref{4.11} will define smooth functions for Hadamard's coefficients on the diagonal.

For $j\in \{0,\dots,N-2\}$ note that 
$$\pd_x^{j} D_N=-\pd_x^{j-N+1}\,(x-y)=-(x-y)\pd_x^{j-N+1}-(j-N+1)\pd_x^{j-N},$$ 
and therefore
$$\pd_x^{j} D_N=-(j-N+1)\pd_x^{j-N}, \qquad \text{ when $y=x$}.$$ 
By induction on $k$, it follows that
\begin{align*}
\pd_x^{j} D_N^{k}&=(-1)^{k}\left(\prod_{\ell=1}^{k}(j-\ell N+1)\right)\pd_x^{j-kN}, && \text{ when $y=x$}\\
&=N^k \left(1-\frac{j+1}{N}\right)_k\,\pd_x^{j-kN},  &&\text{ when $y=x$,}\\
\end{align*}
where 
$$(a)_k = a(a+1)\ldots(a+k-1)$$ 
is the Pochhammer symbol. Thus, for $y=x$, equation \eqref{4.9} yields
\begin{equation}\label{4.13}
H (x,x,\pd_x)= 1+ \sum_{j=0}^{N-2}\sum_{k=1}^{\infty}H_k^{j}(x,x) \ka_N^k  \left(1-\frac{j+1}{N}\right)_k\,\pd_x^{j-kN}.
\end{equation}
Note that the values $H_k^{j}(x,x)$ of Hadamard's coefficients on the diagonal are determined uniquely from the coefficients of $\pd_x^{-n}$ in the operator $H (x,x,\pd_x)$ for $n\not\equiv 1\mod N$, while the coefficients of $\pd_x^{-n}$ with  $n\equiv 1\mod N$ are zero.  Thus, if we can show that in the expansion in \eqref{4.12}, or equivalently in \eqref{4.6}, we have 
\begin{equation}\label{4.14}
\om_n(x,x)=0 \qquad \text { when }n\equiv 1 \mod N,
\end{equation}
it will follow that for $y=x$, the right-hand side of \eqref{4.11} will fix smooth values of the coefficients $H_k^{j}(x,y)$ on the diagonal.
For arbitrary pseudo-differential operators $P$ and $Q$, we have 
\begin{equation}\label{4.15}
\res_z\left[(P e^{xz})\, (Q^{*} e^{-xz})\right]=\res_{\pd_x}PQ,
\end{equation}
see for instance \cite[page~92, 6.2.5]{Dickey}. Using this formula and equations \eqref{4.4}-\eqref{4.6} we see that for $\ell\in\Nset_0$ we have
\begin{align*}
\om_{\ell N+1}(x,x)&=\res_z\left[z^{\ell N} \Psi(x,z) \Psi^*(x,z)\right]=\res_z\left[(\cL^{\ell} S e^{xz} )(S^{-1})^*e^{-xz}\right]\\
&=\res_{\pd_x}\left[\cL^{\ell} S S^{-1}\right]=\res_{\pd_x}\left[\cL^{\ell}\right]=0,
\end{align*}
showing \eqref{4.14} and completing the proof.
\end{proof}

As an immediate corollary of \eqref{4.11}, we obtain an explicit formula for Hadamard's coefficients in terms of the coefficients $\om_n(x,y)$ in the resolvent.

\begin{Theorem}\label{th4.4}
For $k\in\Nset$ and $j\in\{0,\dots,N-2\}$ we have
\begin{equation}\label{4.16}
H_k^{j}(x,y)=(-\ka_N N)^k \sum_{n=1}^{k(N-1)-j}\om_n(x,y)\; \res_{\pd_x}\left[\pd_x^{-n}\left(\frac{1}{x-y}\pd_x^{N-1}\right)^k\pd_x^{-j-1}\right].
\end{equation}
\end{Theorem}

\begin{proof}
For $j,\ell\in\{0,1,\dots,N-2\}$ and $m\in\Zset$ it is easy to see that 
\begin{equation*}
\pd_x^{\ell}D_N^{m}\pd_x^{-j-1}=
\begin{cases}
\text{a pseudo-differential operator of the form }\sum_{i\leq -2}r_i\pd_x^{i} & \text{ if }m>0,\\
\text{a differential operator} & \text{ if }m<0,
\end{cases}
\end{equation*}
hence 
\begin{equation*}
 \res_{\pd_x}\left[\pd_x^{\ell}D_N^{m}\pd_x^{-j-1}\right]=\delta_{m0}\delta_{j\ell}.
\end{equation*}
The proof follows immediately by combining the last equation with \eqref{4.9} and \eqref{4.11}.
\end{proof}

The formula in \thref{th4.4} can be useful in practice when the coefficients $\om_n(x,y)$ can be computed explicitly and, in particular, for operators $\cL$ parametrized by the points of Sato-Segal-Wilson Grassmannian $\Gr$, which include the operators corresponding to the soliton solutions and the more general algebro-geometric solutions of the Gelfand-Dickey hierarchy. We explore these connections in the next sections. Note that if we expand $\left(\frac{1}{x-y}\pd_x^{N-1}\right)^k$, pull $\pd_x$ to the right and compute the residue in \eqref{4.16} we obtain a formula for Hadamard's coefficients of the form
\begin{equation}\label{4.17}
H_k^{j}(x,y)=\frac{\sum_{n=1}^{k(N-1)-j}b_{k,j,n}(x-y)^{n-1}\,\om_{n}(x,y)}{(x-y)^{kN-j-1}},
\end{equation}
where the coefficients $b_{k,j,n}$ depend on $k,j,n$ and the fixed quantities $N$ and $\ka_N$. We perform these calculations in the classical case $N=2$ in the next example which leads  to a new derivation of the formula discovered in \cite{I05}.

\begin{Example}
We fix in this example $N=2$ and therefore $\ka_2=1$. By induction on $k$ it is easy to see that 
\begin{equation*}
\left(\frac{1}{x-y}\pd_x\right)^k=(-1)^k\sum_{\ell=0}^{k-1}\frac{(-1)^{\ell+1}(k-\ell)_{k-1}}{2^{k-\ell-1}\,\ell! }\frac{1}{(x-y)^{2k-\ell-1}}\,\pd_x^{\ell+1}.
\end{equation*}
Plugging this into \eqref{4.16} with $j=0$ yields 
\begin{equation}\label{4.18}
H_{k}^{0}(x,y)= \sum_{n=1}^{k}\frac{(-1)^n\,\om_{n}(x,y)}{(x-y)^{2k-n}} \sum_{\ell=n-1}^{k-1}2^{\ell+1}\,\frac{(k-\ell)_{k-n+\ell}}{\ell!} \binom{-n}{-n+\ell+1}.
\end{equation}
Note that 
$$\frac{(k-\ell)_{k-n+\ell}}{\ell!} \binom{-n}{-n+\ell+1}=\frac{(n)_{2k-2n}}{(k-n)!}\, (-1)^{\ell-n+1}\binom{k-n}{\ell-n+1},$$
and if we change the summation index by setting $\ell=n-1+j$, we obtain
\begin{align*}
 \sum_{\ell=n-1}^{k-1}2^{\ell+1}\,\frac{(k-\ell)_{k-n+\ell}}{\ell!} \binom{-n}{-n+\ell+1}&=2^{n}\,\frac{(n)_{2k-2n}}{(k-n)!} \sum_{j=0}^{k-n}(-2)^j\binom{k-n}{j} \\
 &=2^{n}\,\frac{(n)_{2k-2n}}{(k-n)!} (1-2)^{k-n}.
\end{align*}
Substituting the last equation in \eqref{4.18}, we see that 
\begin{equation}\label{4.19}
H_{k}^{0}(x,y)=(-1)^k\, \sum_{n=1}^{k} \frac{2^{n}(n)_{2k-2n}}{(k-n)!}\,\frac{\om_{n}(x,y)}{(x-y)^{2k-n}},
\end{equation}
leading to a new proof of \cite[Theorem~3.1]{I05}.
\end{Example}

From \eqref{4.6}, it follows that the resolvents of the operators $\cL$ and $(-1)^N\cL^*$ can be related if we exchange the roles of $x$ and $y$, and if we replace $z$ by $-z$, i.e.
\begin{equation} \label{4.20}
\fR(\cL;x,y,z)=\fR((-1)^N\cL^*;y,x,-z).
\end{equation}

Combining the last formula with the explicit connection between the resolvent and the heat kernel, we see that exchanging the roles $x$ and $y$ relates Hadamard's coefficients of the heat equation \eqref{3.2} to Hadamard's coefficients of the dual heat equation
\begin{equation}\label{4.21}
\pd_t v = \ka_N \cL^{*}v.
\end{equation}

\begin{Corollary}\label{Co4.6}
Hadamard's coefficients $H_n^{j}(x,y)$ of the heat equation \eqref{3.2} are related to Hadamard's coefficients $\tilde{H}_n^{j}(x,y)$ of  equation \eqref{4.21} via
\begin{equation}\label{4.22}
H_n^{j}(x,y)=(-1)^{j}\tilde{H}_n^{j}(y,x).
\end{equation}
\end{Corollary}
In particular, if $N$ is even and $\cL$ is formally self-adjoint, \eqref{4.22}  extends the well-known symmetry property of Hadamard's coefficients about the diagonal for the Schr\"{o}dinger operator discussed in Example~\ref{Ex3.2}. If $N$ is odd and $\cL$ is skew-symmetric, we can use Remark~\ref{re3.4} to deduce that the coefficients $H_n^{j}(x,y)$ are symmetric or skew-symmetric depending on the parity of $n+j$. We summarize these properties below.
\begin{Corollary}\label{Co4.7}
\begin{enumerate}[(i)]
\item If $N$ is even and the operator $\cL$ is formally self-adjoint, i.e. $\cL^*=\cL$, then 
\begin{equation}\label{4.23}
H_n^{j}(x,y)=(-1)^{j}H_n^{j}(y,x).
\end{equation}
\item If $N$ is odd and the operator $\cL$ is formally skew-symmetric, i.e. $\cL^*=-\cL$, then 
\begin{equation}\label{4.24}
H_n^{j}(x,y)=(-1)^{n+j}H_n^{j}(y,x).
\end{equation}
\end{enumerate}
\end{Corollary}

\section{Hadamard's coefficients and the Gelfand-Dickey hierarchy} \label{se5}

Recall that if $\cL=\pd_x^N+u_{N-2}\pd_x^{N-2}+\cdots+u_0$ is an $N$-th order differential operator, the $N$-th {\em Gelfand-Dickey hierarchy} is the set of equations
\begin{equation}\label{5.1}
\frac{\pd \cL}{\pd s_m}=\left[(\cL^{m/N})_+,\cL \right].
\end{equation} 
For every fixed $m\in\Nset$, equation \eqref{5.1} can be rewritten as a system of partial differential equations for the coefficients $u_0,\dots,u_{N-2}$. The vector fields $\pd/\pd s_m$ commute with each other, and therefore each of equations \eqref{5.1} defines symmetries for all other equations. Moreover, the functionals
\begin{equation}\label{5.2}
\int \res_{\pd_x}\cL^{k/N}\,dx
\end{equation}
are first integrals of all equations of the Gelfand-Dickey hierarchy, see for instance \cite[Proposition~1.5.3]{Dickey}. When $N=2$ and $N=3$, the  Gelfand-Dickey hierarchy \eqref{5.1} reduces to the KdV  and the Boussinesq hierarchies, respectively.

The proof of \thref{th4.2} shows that the values of Hadamard's coefficients on the diagonal $y=x$ differ by constant factors from the coefficients $\om_n(x,x)$. More precisely, from \eqref{4.11} and \eqref{4.13} we see that the following relation holds
\begin{equation}\label{5.3}
H_k^{j}(x,x)=\frac{\ka_N^k}{(1-(j+1)/N)_k} \om_{kN-j}(x,x), \quad \text{ for }j\in\{0,1,\dots,N-2\}, \quad k\in\Nset.
\end{equation}

Using the last formula, it is not difficult to see that the values of Hadamard's coefficients on the diagonal also provide first integrals of the hierarchy \eqref{5.1}.
\begin{Theorem}\label{th5.1}
We have
\begin{equation}\label{5.4}
H_k^{j}(x,x)=\frac{\ka_N^k}{(1-(j+1)/N)_k}  \res_{\pd_x}\cL^{k-\frac{j+1}{N}}, \quad \text{ for }j\in\{0,1,\dots,N-2\}, \quad k\in\Nset.
\end{equation}
Therefore, the functionals 
\begin{equation}\label{5.5}
J_{k,j}=\int H_k^{j}(x,x) dx, \quad \text{ for }j\in\{0,1,\dots,N-2\}, \quad k\in\Nset
\end{equation}
are first integrals of all equations of the Gelfand-Dickey hierarchy \eqref{5.1}.
\end{Theorem}

\begin{proof}
From \eqref{4.2} and \eqref{4.4} it follows that for $m\in\Nset$ we have
\begin{align}
\cL^{m/N}=S\pd_x^m ((S^{*})^{-1})^*&= \left(\sum_{k=0}^{\infty}\psi_k(x) \pd_x^{-k}\right)\pd_x^{m}  \left(\sum_{j=0}^{\infty}\psi_j^{*}(x) (-\pd_x)^{-j}\right)^{*}\nonumber\\
&= \sum_{j,k=0}^{\infty} \psi_k(x) \pd_x^{m-(k+j)}\, \psi_j^{*}(x),\label{5.6}
\end{align}
and therefore 
\begin{equation}\label{5.7}
\res_{\pd_x}\cL^{m/N} = \sum_{k+j=m+1}\psi_k(x)  \psi_j^{*}(x)= \om_{m+1}(x,x).
\end{equation}
Plugging $m=kN-(j+1)$ in the last formula and using \eqref{5.3} we obtain \eqref{5.4}.
\end{proof}

\begin{Remark}[KdV hierarchy]\label{re5.2}
When $N=2$ and $\cL=\pd_x^2+u_0$, a short computation shows that 
\begin{equation*}
\left[(\cL^{(2m-1)/2})_+,\cL \right]=2\,\frac{\pd\, \om_{2m}(x,x)}{\pd x},
\end{equation*}
see \cite[page 2125]{I05}. Combining this with \eqref{5.3} we see that the KdV hierarchy \eqref{5.1} can be rewritten in terms of Hadamard's coefficients as follows
\begin{equation*}
\frac{\pd u_0}{\pd s_{2m-1}}=\frac{(2m-1)!!}{2^{m-1}} \,\frac{\pd H_{m}^{0}(x,x)}{\pd x}, \qquad \text{ where }m\in\Nset.
\end{equation*}
This representation has been used by several authors to find explicit formulas for the higher KdV flows using Hadamard's coefficients, see for instance \cite{ASch} and the references therein. When $N>2$, the right-hand sides of the Gelfand-Dickey equations \eqref{5.1} can no longer be expressed just as scalar multiples of the derivatives of the resolvent coefficients on the diagonal, and the connection becomes more subtle. However, one can use the explicit relation in \thref{th4.2} to generate the Gelfand-Dickey flows by taking more terms in the expansion of Hadamard's coefficients about the diagonal $y=x$. We illustrate this in the case when $N=3$ which corresponds to the Boussinesq hierarchy.
\end{Remark}

\begin{Proposition}[Boussinesq hierarchy]\label{pr5.3}
Let $N=3$, $\ka_3=1$, 
$$\cL=\pd_x^3+u_1(x)\pd_x+u_0(x),$$ 
and let $\{H_k^{j}(x,y)\}_{k\in\Nset, \; j\in\{0,1\}}$ denote the Hadamard's coefficients in the expansion of the heat kernel \eqref{3.7}. The Boussinesq hierarchy \eqref{5.1} is equivalent to the equations
\begin{subequations}\label{5.8}
\begin{align}
\frac{\pd u_1}{\pd s_{3k-2}}=&3 (1/3)_{k}\, \frac{\pd }{\pd x} H_k^{1}(x,x)\nonumber\\
\frac{\pd u_0}{\pd s_{3k-2}}=&3 (1/3)_{k}\, \frac{\pd }{\pd x} \left[H_k^{0}(x,x)+ \frac{\pd H_k^{1}(x,y)}{\pd x}\Bigg|_{y=x} \right]\label{5.8a}
\intertext{and} 
\frac{\pd u_1}{\pd s_{3k-1}}=&3 (2/3)_{k}\, \frac{\pd }{\pd x} H_k^{0}(x,x)\nonumber\\
\frac{\pd u_0}{\pd s_{3k-1}}=&3 (2/3)_{k}\, \frac{\pd }{\pd x} \left[ \frac{\pd H_k^{0}(x,y)}{\pd x}\Bigg|_{y=x} \right],\label{5.8b}
\end{align}
\end{subequations}
where $k\in\Nset$.
\end{Proposition}

\begin{proof}
Using \eqref{5.6} with $N=3$ we see that 
$$(\cL^{m/3})_{-}=r_{1}(x)\pd_x^{-1}+r_2(x)\pd_x^{-2}+\cdots,$$
where 
$$r_1(x)=\om_{m+1}(x,x)\qquad\text{ and }\qquad r_2(x)=\om_{m+2}(x,x)-\frac{\pd \om_{m+1}(x,y)}{\pd y}\Bigg|_{y=x}.$$
On the other hand,
\begin{align}
\left[(\cL^{m/3})_+,\cL \right]&= \left[\cL ,(\cL^{m/3})_-\right]\nonumber\\
&=\left[\pd_x^3+u_1(x)\pd_x+u_0(x),  r_{1}(x)\pd_x^{-1}+r_2(x)\pd_x^{-2}\right]_{+}\nonumber\\
&=3r_1'(x)\pd_x+3(r_1'(x)+r_2(x))'\nonumber\\
&=3 (\om_{m+1}(x,x))'\pd_x+ 3\left(\om_{m+2}(x,x) +\frac{\pd \om_{m+1}(x,y)}{\pd x}\Big|_{y=x}\right)'.\label{5.9}
\end{align}
From \eqref{5.3} we know that 
$$\om_{3k-1}(x,x)=(1/3)_kH_k^1(x,x)\qquad\text{ and }\qquad\om_{3k}(x,x)=(2/3)_kH_k^0(x,x),$$
which gives the formulas for the flows on $u_1$ in \eqref{5.8} when we take $m=3k-2$ and $m=3k-1$, respectively. To obtain the second equations, we differentiate \eqref{4.11} with respect to $x$ and then we set $y=x$. Comparing the coefficients of $\pd_x^{-3k}$ in the resulting equation we see that 
$$\frac{\pd \om_{3k}(x,y)}{\pd x}\Big|_{y=x}= (2/3)_{k}\, \frac{\pd H_k^{0}(x,y)}{\pd x}\Big|_{y=x},$$
which combined with formula \eqref{5.9} for $m=3k-1$ and the fact that $\om_{3k+1}(x,x)=0$ yields the flow on $u_0$ in \eqref{5.8b}. 
The proof of the second equation in formula \eqref{5.8a} follows along the same lines.
\end{proof}

\begin{Remark}
The explicit formulas for Hadamard's coefficients about the diagonal in Example~\ref{Ex3.3} can be used to obtain the first two flows in the Boussinesq hierarchy by taking $k=1$ in \eqref{5.8}. The first flow corresponds to the simple translation $x\to x+s_1$, but the second one leads to the Boussinesq equation which is often used as a name for the Gelfand-Dickey hierarchy when $N=3$.
\end{Remark}

\section{Finite heat kernel expansions} \label{se6}

Recall that for every differential operator $\cL$, the set $\cA$ consisting of all differential operators commuting with $\cL$ is a commutative ring \cite{Schur}. An important invariant of $\cA$ is its rank, which is the greatest common divisor of the orders of the operators in $\cA$. In particular, if $\cA$ is of rank $1$, then $\mathrm{Spec}(\cA)$ is an irreducible algebraic curve that can be completed by adding one smooth point at infinity  \cite{BC,Krichever,Mumford}. 

In this section, we prove that the heat kernel consists of finitely many terms, i.e. only finitely many of Hadamard's coefficients are nonzero, if and only if the operator $\cL=\pd_x^N+u_{N-2}(x)\pd_x^{N-2}+\cdots+u_0(x)$  belongs to a rank-one commutative ring of differential operators whose spectral curve $\mathrm{Spec}(\cA)$ is rational with only one cusp-like singular point, and the coefficients $u_j(x)$ vanish at $\infty$. These operators can be conveniently parametrized by a sub-Grassmannian $\Gr_0$ of the Sato-Segal-Wilson Grassmannian $\Gr$ \cite{DJKM,SS,SW}. We start by reviewing briefly the construction of $\Gr$ together with the Burchnall-Chaundy-Krichever correspondence \cite{BC,Krichever,Mumford} between commutative rings of differential operators and algebraic curves.

Let $S^1=\{z\in\Cset:|z|=1\}$ denote the unit circle and let $\sH=L^2(S^1,\Cset)$ be the Hilbert space of $L^2$ functions on $S^1$. We can split $\sH$ as the orthogonal 
direct sum $\sH=\sH_+\oplus \sH_-$, where $\sH_+$ and  $\sH_-$ consist of the functions whose Fourier series involve only nonnegative and only negative
powers of $z$, respectively. The Sato-Segal-Wilson Grassmannian $\Gr$ consists of all closed subspaces $W$ of $\sH$ such that
\begin{enumerate}[(i)]
\item the projection $W\rightarrow \sH_+$ is a Fredholm operator of index zero;
\item the projection $W\rightarrow \sH_-$ is a compact operator.
\end{enumerate}
For each $W\in\Gr$ there is a unique Baker function $\Psi_W(x,z)$ 
characterized by the following properties:
\begin{enumerate}[(i)]
\item $\Psi_W$ has the form 
$\Psi_W(x,z)=\left(1+\sum_{i=1}^{\infty}\psi_i(x)z^{-i}\right)e^{xz}$;
\item $\Psi_W(x,.)$ belongs to $W$ for each $x$.
\end{enumerate}
We  denote by $W^{\alg}$ the dense subspace of  elements of finite order, i.e. 
the elements of the form $\sum_{j\leq k}a_jz^j$. Let $A_W$  be the ring of 
analytic functions $f(z)$ on $S^1$ that leave $W^{\alg}$ invariant, i.e.
\begin{equation}\label{6.1}
A_W=\{ f(z)\text{ is analytic on }S^1:f(z)W^{\alg}\subset W^{\alg}\}.
\end{equation}
For each $f(z)\in A_W$ there exists a unique differential operator $\cK_f=\cK_f(x,\pd_x)$ in the variable $x$, such that
\begin{equation}\label{6.2}
\cK_f(x,\pd_x)\Psi_W(x,z)=f(z)\Psi_W(x,z).
\end{equation} 
The order of the operator $\cK_f$ is equal to the order of $f$. In general, 
$A_W$ is trivial, i.e. $A_W=\Cset$. The spaces $W$ which give rise to rank-one commutative rings of differential operators are those for which $A_W$ contains an element of each sufficiently large order, and in this case the ring of operators is obtained by taking
\begin{equation}\label{6.3}
\cA_W=\{\cK_f:f\in A_W\}.
\end{equation}
Clearly, $\cA_W$ is isomorphic to $A_W$ and has rank 1. Conversely, up to a change of variable and gauge transformation, every maximal rank-one commutative ring of differential operators can be obtained by this construction for appropriate $W\in\Gr$.

We denote by $\Gr^{(N)}$ the subspace of $\Gr$ consisting of those $W$ which are invariant under the multiplication by $z^N$, i.e. $z^N\in A_W$. Note that if $W\in\Gr^{(N)}$ then with the notation in \eqref{6.2} it follows that $\cK_{z^N}$ is a monic differential operator of order $N$ which belongs to the commutative ring $\cA_W$, and we can identify the Baker function $\Psi_W(x,z)$ with the formal wave function of the operator $\cL=\cK_{z^N}$ introduced in \seref{se4}.

Let $\Gr_0$ be the subspace of $\Gr$ consisting of those $W$ such that
\begin{equation}\label{6.4}
z^{k}\sH_+\subset W\subset z^{-k}\sH_+, \qquad \text{ for some }k\in\Nset_0,
\end{equation}
and we set 
$$\Gr_0^{(N)}=\Gr_0\cap \Gr^{(N)}.$$
The sub-Grassmannian $\Gr_0$ parametrizes the rank-one commutative rings of differential operators whose spectral curves are rational with just one cusp-like singularity \cite[Proposition 7.3]{SW}.

With these notations, we can formulate the main result in this section, which characterizes the operators $\cL$ for which the heat-kernel consists of finitely many terms.

\begin{Theorem}\label{th6.1} 
For an $N$-th order operator $\cL=\pd_x^N+u_{N-2}(x)\pd_x^{N-2}+\cdots+u_0(x)$ the following conditions are equivalent.
\begin{enumerate}[\rm(i)]
\item Finitely many of Hadamard's coefficients $H_k^{j}(x,y)$ are not identically zero.
\item The operator $\cL$  belongs to a rank-one commutative ring $\cA$ of differential operators whose spectral curve $\mathrm{Spec}(\cA)$ is rational with only one cusp-like singular point, and the coefficients $u_j(x)$ vanish at $x=\infty$.
\item There exists $W\in\Gr_0^{(N)}$ such that $\cL=\cK_{z^N}$.
\end{enumerate}
\end{Theorem}

\begin{proof}[Proof of the equivalence {\rm{(ii)}$\Leftrightarrow$\rm{(iii)}} in \thref{th6.1}]
The equivalence of (ii) and (iii) above follows easily from the statements in \cite{SW}. Indeed, if (iii) holds, then we know that $\cL$  belongs to a rank-one commutative ring of differential operators whose spectral curve $\mathrm{Spec}(\cA)$ is rational with only one cusp-like singular point, and we just have to show that  the coefficients $u_j(x)$ vanish at $\infty$. 
This follows easily by using the notion of a tau-function and the fact that $W\in\Gr_0$ if and only if the tau-function $\tau_W(s)$ is a polynomial in a finite number of the time variables $s=(s_1,s_2,\dots)$ for the KP hierarchy,  see \cite[Proposition 8.5]{SW}. 
More precisely, the wave function $\Psi(x,z)=\Psi_W(x,z)$ and the adjoint wave function $\Psi^*(x,z)=\Psi_W^*(x,z)$ can be expressed in terms of a tau-function $\tau(s)=\tau_W(s)$ associated with $W\in\Gr$ by Sato's formulas
\begin{subequations}\label{6.5}
\begin{align}
\Psi(x,z)&=\frac{\tau(s-[z^{-1}])}{\tau(s)}e^{xz} \Big|_{s_1=x,\,s_2=s_3=\cdots=0} \label{6.5a}\\
\intertext{ and }
\Psi^*(x,z)&=\frac{\tau(s+[z^{-1}])}{\tau(s)}e^{-xz} \Big|_{s_1=x,\,s_2=s_3=\cdots=0}\label{6.5b}
\end{align}
\end{subequations}
where $[z]=(z,z^2/2,z^3/3,\dots)$. If $\tau(s_1,s_2,\dots)$ is a polynomial in a finite number of the time variables $s_1,s_2,\dots$, then $\Psi(x,z)$ and $\Psi^*(x,z)$ will have finite expansions in a neighborhood of $z=\infty$:
\begin{subequations}\label{6.6}
\begin{align}
\Psi(x,z)&=\left(1+\sum_{k=1}^{m_1}\frac{\psi_k(x)}{ z^{k}}\right)e^{xz}, &&\text{ and }\lim_{x\to\infty}\psi_k(x)=0 \text{ for }k>0, \label{6.6a}\\
\Psi^*(x,z)&=\left(1+\sum_{k=1}^{m_2}\frac{\psi_k^*(x)}{ z^{k}}\right)e^{-xz},  &&\text{ and }\lim_{x\to\infty}\psi_k^*(x)=0 \text{ for }k>0. \label{6.6b}
\end{align}
\end{subequations}
From formula \eqref{5.6} with $m=N$, we see that
\begin{equation*}
\cL= \sum_{k_1=0}^{m_1} \sum_{k_2=0}^{m_2}  \psi_{k_1}(x) \pd_x^{N-(k_1+k_2)}\, \psi_{k_2}^{*}(x),
\end{equation*}
and using \eqref{6.6} it follows that the coefficients $u_j(x)$ vanish at $\infty$. Conversely, if (ii) holds, then from  \cite[Proposition 7.3]{SW} we conclude that there exists $W\in\Gr_0^{(N)}$ and a polynomial $f(z)=z^N+f_{N-2}z^{N-2}+\cdots+ f_0\in A_{W}$ such that $\cL=\cK_{f}$, and equation \eqref{6.2} holds. Therefore, using the same notations as above, it follows that 
\begin{equation*}
\cL= \sum_{k_1=0}^{m_1} \sum_{k_2=0}^{m_2}  \psi_{k_1}(x) f(\pd_x)\pd_x^{-k_1-k_2}\, \psi_{k_2}^{*}(x),
\end{equation*}
From the last formula and \eqref{6.6} it is easy to see that $\lim_{x\to\infty}u_j(x)=f_j$, and therefore $f_0=f_1=\cdots=f_{N-2}=0$, which means that $f(z)=z^N$ and $\cL=\cK_{z^N}$, completing the proof of (iii).
\end{proof}

Before we prove that (i) is equivalent to (iii), we establish another useful characterization of the sub-Grassmannian $\Gr_0$, which can be deduced from the work of Burchnall and Chaundy [\ref{BC}(c)].

\begin{Lemma}\label{le6.2}
For an $N$-th order operator $\cL=\pd_x^N+u_{N-2}(x)\pd_x^{N-2}+\cdots+u_0(x)$ the following conditions are equivalent.
\begin{enumerate}[\rm(i)]
\item The coefficients $\om_n(x,y)$ in $\fR(\cL)$ vanish for $n$ sufficiently large.
\item There exists a monic differential operator $\cV$ such that $\cL=\cV \pd_x^N\cV^{-1}$ and $\ker(\cV)\subset \Cset[x]$.
\item There exists $W\in\Gr_0^{(N)}$ such that $\cL=\cK_{z^N}$.
\end{enumerate}
\end{Lemma}

\begin{proof}[Proof of \leref{le6.2}] 
Suppose first that (i) holds, i.e.
$$\om_n(x,y)=0 \quad \text { for }\quad n > n_0.$$ 
From the definition of $\om_n(x,y)$ in \eqref{4.6}, it follows that
\begin{equation}\label{6.7}
\res_{z}\left[z^n\Psi(x,z)\, \pd_x^{j}\Psi^{*}(x,z) \right]=0,\qquad \text{ for all }j\in\Nset_0\text{ and }n\geq n_0.
\end{equation}
From \eqref{4.4} we see that 
$$z^n\Psi(x,z) =S\pd_x^n e^{xz} \qquad \text{and }\qquad \pd_x^{j}\Psi^{*}(x,z)= \pd_x^{j}(S^{*} )^{-1} e^{-xz},$$
which combined with \eqref{6.7} and \eqref{4.15} shows that
\begin{equation*}
\res_{\pd_x}\left[ (S\pd_x^nS^{-1})(-\pd_x)^{j} \right]=0,\qquad \text{ for all }j\in\Nset_0,\text{ and }n\geq n_0.
\end{equation*}
This means that for every $n\geq n_0$, $\cM_n=S\pd_x^nS^{-1}$ is a differential operator of order $n$, commuting with $\cL$. In particular, if we take $n\geq n_0$ which is relatively prime with $N$, it follows that $\cM_n^N=\cL^n$. Since the operators $\cM_n$ and $\cL$ have relatively prime orders, [\ref{BC}(c), Theorem 3, p.~473] tells us that there exist monic differential operators $\cV$ and $\cU$, and a positive integer $M$ such that  
\begin{equation}\label{6.8}
\cL\cV=\cV \pd_x^N,\quad\text{ and }\quad\cU\cV=\pd_x^M,
\end{equation}
from which (ii) follows. 

If (ii) holds and if $m$ denotes the order of $\cV$, we use $S=\cV\pd_x^{-m}$ as a wave operator and we define the Baker function by
\begin{equation}\label{6.9}
\Psi_W(x,z)=\frac{1}{z^m}\cV(e^{xz}).
\end{equation}
We can recover $W\in\Gr$ by considering the closed span of the functions $\Psi_W(x,z)$ as $x$ varies. In particular, if $\Psi_W(x,z)$ is regular at $x=x_0$, we can take the closed span of the functions $\pd_x^{j}\Psi_W(x,z)|_{x=x_0}$ where $j=0,1,\dots$, and if $x_0=0$, these functions will form a basis for the dense subspace $W^{\mathrm{alg}}$ of $W$. From \eqref{6.9} we see that 
\begin{equation}\label{6.10}
W\subset z^{-m}\sH_+. 
\end{equation}
Since $\ker(\cV)\subset \Cset[x]$, there exists a positive integer $M$ such that $\ker(\cV)\subset \ker(\pd_x^M)$ and therefore we can find an operator $\cU$ such that $\cU\cV=\pd_x^M$, i.e. \eqref{6.8} holds. Using this and applying $\cU$ to \eqref{6.9} we see that $\cU \Psi_W(x,z)=z^{M-m}e^{xz}$, which shows that 
\begin{equation}\label{6.11}
z^{M-m}\sH_+\subset W.
\end{equation}
From equations \eqref{6.10}-\eqref{6.11} it follows that $W\in\Gr_0^{(N)}$, thus completing the proof of (iii). 

Finally, if (iii) holds then $\tau_W(s_1,s_2,\dots)$ is a polynomial in a finite number of the time variables $s_1,s_2,\dots$, and formulas \eqref{6.5} show that $\Psi(x,z)$ and $\Psi^*(x,z)$ will have finite expansions as in \eqref{6.6} in a neighborhood of $z=\infty$, establishing (i) and completing the proof of the lemma.
\end{proof}

\begin{proof}[Proof of the equivalence {\rm{(i)}$\Leftrightarrow$\rm{(iii)}} in \thref{th6.1}]
Suppose first that (i) holds, i.e. there exists $k_0\in\Nset$ such that
\begin{equation}\label{6.12}
H_k^{j}(x,y)=0 \qquad \text{ for }j\in\{0,\dots,N-2\}\text{ and } k\geq k_0.
\end{equation}
From \eqref{4.10} it follows that for $n\in\Nset_0$ we have
\begin{equation}\label{6.13}
D_N^n=\sum_{k=0}^{n}(-1)^{n-k}a_{n,k}(x-y)^{n-k}\pd_x^{-n(N-1)-k},
\end{equation}
where the coefficients $a_{n,k}$ depend only on $n,k$ and $N$, but for simplicity we omit the $N$-dependence since $N$ is fixed.
If we compare the coefficients of $\pd_x^{-n}$ in formula \eqref{4.11}, and using the expansion in \eqref{6.13}, it follows that we can write $\om_n(x,y)$ as a linear combination of $H_k^{j}(x,y)$,  where $k$ is such that 
$$\frac{n}{N}\leq k<\frac{n}{N-1}+1,$$
with coefficients involving powers of $(x-y)$. This shows that $\om_n(x,y)$ will vanish identically for $n$ sufficiently large, i.e.
$$\om_n(x,y)=0 \quad \text { for }\quad n\geq n_0.$$ 
For instance, with the notations above we can take $n_0= Nk_0$. Applying \leref{le6.2} we see that (iii) holds, completing the proof in this direction.

Conversely, suppose now that (iii) holds. We know from \leref{le6.2} that $\om_n(x,y)=0$ for $n$ sufficiently large, say $n\geq n_0$. 
We will use again the fact that if $W\in\Gr_0$ then the tau-function $\tau_W(s)$ is a polynomial in a finite number of the time variables $s=(s_1,s_2,\dots)$ and therefore equations \eqref{6.6} hold. Sato's formulas \eqref{6.5} show that the nonzero coefficients $\om_n(x,y)$ can be written as 
$$\om_n(x,y)=\frac{\tilde{\om}_n(x,y)}{\tilde{\tau}(x)\tilde{\tau}(y)},\qquad \text{ for }n<n_0,$$
where $\tilde{\tau}(x)=\tau(x,0,0,\dots)$ is a polynomial of $x$, and $\tilde{\om}_n(x,y)$ are polynomials of $x$ and $y$ for all $n<n_0$.
Substituting this into \eqref{4.17} yields
\begin{equation}\label{6.14}
H_k^{j}(x,y)=\frac{\sum_{n=1}^{n_0-1}b_{k,j,n}(x-y)^{n-1} \,\tilde{\om}_{n}(x,y)}{\tilde{\tau}(x)\tilde{\tau}(y)\,(x-y)^{kN-j-1}},
\end{equation}
where $b_{k,j,n}$ are numbers depending on $k,j,n$. Note that the numerator is a polynomial of $x$ and $y$ of total degree at most $m$, where 
\begin{equation}\label{6.15}
m=\max_{0<n<n_0}[n-1+\deg \tilde{\om}_{n}(x,y)]
\end{equation} 
and $\deg \tilde{\om}_{n}(x,y)$ denotes the total degree of the polynomial $\tilde{\om}_{n}(x,y)$. Since $m$ is a constant depending only on the operator $\cL$, while the power of $(x-y)$ in the denominator in \eqref{6.14} grows linearly with $k$, the smoothness of Hadamard's coefficients in a neighborhood of the diagonal $y=x$ implies that they must vanish identically when $k$ is sufficiently large. For instance, with the above notations, it follows that 
\begin{equation}\label{6.16}
H_k^{j}(x,y)=0, \qquad\text{ when }\quad kN -j-1\geq m+1.
\end{equation} 
In particular, since $0\leq j\leq N-2$, it follows that
\begin{equation}\label{6.17}
H_k^{j}(x,y)=0,\qquad\text{ when }(k-1)N\geq m,
\end{equation} 
completing the proof.
\end{proof}

\begin{Remark}\label{Re6.3}
It is interesting to note that conditions (ii) and (iii) in \thref{th6.1} are equivalent to the condition that $\om_n(x,y)=0$ for $n$ sufficiently large, but it does not seem easy to use the latter fact and formula \eqref{4.16}  to prove directly that Hadamard's coefficients $H_k^{j}(x,y)$ also vanish for $k$ sufficiently large, without going through the $\tau$-function. The problem is that, even though only finitely many of the $\om_n$'s are nonzero, the sum on the right-hand side of \eqref{4.16} always starts from $n=1$, so we will always have a few nonzero terms. As a corollary from \thref{th6.1}, we see that for $W\in\Gr_0$, the nonzero coefficients $\om_n(x,y)$ have the exact form needed to get $0$ on the right-hand side of \eqref{4.16} when $k$ is sufficiently large. Earlier proofs \cite{Ha2,I07} of \thref{th6.1} in the case $N=2$ use different properties of the rational KdV potentials and the corresponding spectral curves $w^2=z^{2k+1}$ to show that the finite sum in \eqref{4.19} vanishes for $k$ sufficiently large, establishing a one-to-one correspondence between the natural cells in $\Gr_0^{(2)}$ under the KdV flows and the index of the first vanishing coefficient in the heat kernel, see \cite[Remark 3.1]{I07}.
\end{Remark}

\begin{Remark}\label{Re6.4}
 When $N=2$, the  potentials $u_0(x)$ for which the Schr\"odinger operator $\cL=\pd_x^2+u_0(x)$ satisfies the equivalent conditions in \thref{th6.1} are precisely the rational solutions of the KdV equation which vanish at $x=\infty$, discovered by Airault, McKean and Moser \cite{AMM}. They can also be obtained by successive applications of the Darboux transformation \cite{Dar} starting with the operator $\pd_x^2$ as shown by Adler and Moser \cite{AM}, and exhaust the Schr\"odinger operators of rank one which possess the bispectral property as proved by Duistermaat and Gr\"unbaum \cite{DG}. Using condition (iii) in \thref{th6.1}, the work of Krichever \cite{Krichever2} on the rational solutions of the KP equation and Wilson's characterization of all rank-one bispectral rings \cite{Wilson}, we establish similar characterizations for the operators $\cL$ satisfying the equivalent conditions in \thref{th6.1}. The precise statements are given in \thref{th8.1}, Proposition~\ref{pr9.1} and \reref{Re9.2}. 
 \end{Remark}

\section{Explicit examples}\label{se7}

Since the case $N=2$ has been studied extensively in the literature, in this section we present a few explicit examples when $N\geq 3$.

\begin{Example}\label{Ex7.1}
Consider the space
$$W=\clspan\{z^{-1},1,z^2,z^3,\dots\}\in \Gr_0^{(3)}\qquad \text{ and note that }\qquad A_W=z^3\Cset[z].$$ 
The corresponding tau-function is the elementary Schur polynomial 
$$\cS_{2}(s_1,s_2)=\frac{s_1^2}{2}+s_2.$$ 
If we set 
\begin{equation}\label{7.1}
\tau(x;s)=\cS_{2}(s_1+x,s_2)=\frac{(s_1+x)^2}{2}+s_2, 
\end{equation}
then the Baker functions can be computed from \eqref{6.5} 
\begin{subequations}\label{7.2}
\begin{align}
\Psi(x,z)&=\left(1-\frac{s_1+x}{\tau(x;s)\,z}\right)e^{xz}
\intertext{ and }
\Psi^*(x,z)&= \left(1+\frac{s_1+x}{\tau(x;s)\,z}+\frac{1}{\tau(x;s)\,z^2}\right)e^{-xz}.
\end{align}
\end{subequations}
Strictly speaking, we should set $s_1=s_2=0$ in the formulas above in order to get the stationary Baker function and its adjoint, but we can leave $s_1$ and $s_2$ as free parameters. Equivalently, we can think of 
$\Psi(x,z)$ and $\Psi^*(x,z)$ as the stationary Baker functions corresponding to the space 
$$W_s=\clspan\left\{\left(\frac{s_1^2}{2}-s_2\right)z-s_1+\frac{1}{z},-s_1z+1,z^2,z^3,\dots\right\}\in \Gr_0^{(3)},$$
which depends on the free parameters $s_1,s_2$.  The space $W_s$ can be obtained from $W$ by applying the reverse KP flows $s_1,s_2$, namely $W_s=e^{-s_1z-s_2z^2}W$. 

If we take now $z^3\in A_W=A_{W_s}$, we can construct a third-order differential operator $\cL=\cK_{z^3}$ via the Krichever's correspondence \eqref{6.2}. A short computation shows that 
\begin{equation}\label{7.3}
\cL=\cK_{z^3}=\pd_x^3+u_1(x)\pd_x+u_0(x),
\end{equation}
where 
\begin{equation}\label{7.4}
u_1(x)=-\frac{3(x^2+2s_1x+s_1^2-2s_2)}{2\tau^2(x;s)} \quad\text{ and }\quad
u_0(x)=-\frac{6s_2(x+s_1)}{\tau^3(x;s)}.
\end{equation}
Using equations \eqref{7.2} we can easily compute the coefficients in the expansion of the resolvent $\fR(\cL)$ defined in \eqref{4.6}, and we find
\begin{subequations}\label{7.5}
\begin{align}
\om_1(x,y)&=\frac{(x-y)(xy+s_1(x+y)+s_1^2-2s_2)}{2\tau(x;s)\tau(y;s)}\\
\om_2(x,y)&=\frac{x^2-2xy-2s_1y-s_1^2+2s_2}{2\tau(x;s)\tau(y;s)}\\
\om_3(x,y)&=-\frac{x+s_1}{\tau(x;s)\tau(y;s)}\\
\om_k(x,y)&=0, \qquad \text{for }k\geq 4.
\end{align}
\end{subequations}
From these formulas, it follows immediately that $m=3$ in \eqref{6.15} and since $N=3$, we conclude from \eqref{6.17} that 
$$
H_k^j(x,y)=0, \qquad \text{for all }k\geq 2 \text{ and }j\in\{0,1\}.
$$
Computing the first two coefficients using  \eqref{4.16}  in \thref{th4.4}  with $k=1$ and $j=0,1$ we see that
\begin{subequations}\label{7.6}
\begin{align}
H_1^1(x,y)&=-\frac{3(xy+s_1(x+y)+s_1^2-2s_2)}{2\tau(x;s)\tau(y;s)}\\
H_1^0(x,y)&=-\frac{3(x+s_1)}{2\tau(x;s)\tau(y;s)}.
\end{align}
\end{subequations}

With the explicit formulas above and using the fact that $\Ai(z)$ satisfies the differential equation $\Ai''(z)=z\Ai(z)$, one can check directly that 
\begin{equation}
\cH(x,y,t)= \frac{1}{\sqrt[3]{3t}}\Ai\left(-\frac{x-y}{\sqrt[3]{3t}}\right)\left(1+H_1^{0}(x,y)t\right) -  \frac{1}{\sqrt[3]{9t^2}}\Ai'\left(-\frac{x-y}{\sqrt[3]{3t}}\right)  H_1^{1}(x,y)t,
\end{equation}
is the heat kernel  of the equation $\pd_t v =\cL v$ for the third-order operator $\cL$ given in \eqref{7.3}-\eqref{7.4}.
\end{Example}

\begin{Remark}\label{Re7.2}
Note that with the notations in Example \ref{Ex7.1}, we have $W_s\in\Gr_0^{(N)}$ for every $N\geq 3$, and the operators $\cK_{z^N}$ will have the same resolvent. Thus, we can use the above computations to obtain examples of $N$-th order operators having heat kernels with finitely many terms. Indeed, if we set 
\begin{equation}\label{7.8}
\cL_N=\cL^{N/3},
\end{equation}
where $\cL$ is the operator in \eqref{7.3}, then $\cL_N=\cK_{z^N}$ is a differential operator for every $N\geq 3$, and  the coefficients of the resolvent $\fR(\cL_N)$ are given in \eqref{7.5}. It is easy to see that if we pick
\begin{subequations}\label{7.9}
\begin{align}
H_1^{N-2}(x,y)&=-\ka_N\,\frac{N(xy+s_1(x+y)+s_1^2-2s_2)}{2\tau(x;s)\tau(y;s)}\\
H_1^{N-3}(x,y)&=-\ka_N\,\frac{N(x+s_1)}{2\tau(x;s)\tau(y;s)},
\end{align}
\end{subequations}
then 
$$\ka_N\,\frac{H_1^{N-2}(x,y)}{N} \pd_x^{N-2} D_N+\ka_N\,\frac{H_1^{N-3}(x,y)}{N} \pd_x^{N-3} D_N=\sum_{k=1}^{3}\om_k(x,y)\pd_x^{-k},$$
because the left-hand side is independent of $N$, as long as $N\geq3$, and is equal to the right-hand side when $N=3$ by the arguments in Example~\ref{Ex7.1}. This combined with \eqref{4.11} means that all coefficients, except the ones given in \eqref{7.9} are identically zero. We can also obtain this directly following the same steps as in Example~\ref{Ex7.1}: equation \eqref{6.17} tells us that all coefficients $H_k^{j}(x,y)$ with $k\geq 2$ are equal to $0$. If we set $k=1$, we can use \eqref{6.16} to deduce that $H_1^{j}(x,y)=0$ when $j\leq N-5$. A direct computation using \eqref{4.16} shows that $H_1^{N-4}(x,y)=0$, and for $j=N-2$ and $j=N-3$ we obtain the formulas in \eqref{7.9}. Therefore 
\begin{align*}
\cH(x,y,t)=& \frac{1}{\sqrt[N]{t}}A_N\left(\frac{x-y}{\sqrt[N]{t}}\right)
+ \frac{1}{\sqrt[N]{t^{N-2}}}A_N^{(N-3)}\left(\frac{x-y}{\sqrt[N]{t}}\right)H_1^{N-3}(x,y)t \\
&+\frac{1}{\sqrt[N]{t^{N-1}}}A_N^{(N-2)}\left(\frac{x-y}{\sqrt[N]{t}}\right)H_1^{N-2}(x,y)t
\end{align*}
is the heat kernel of the equation 
$$\pd_t v =\ka_N\cL_N v$$ 
for the $N$-th order operator $\cL_N$ in \eqref{7.8}.
\end{Remark}

\begin{Example}\label{Ex7.3}
We can apply Corollary~\ref{Co4.6} to Example~\ref{Ex7.1} and Remark~\ref{Re7.2} to obtain another set of operators having finitely many nonzero Hadamard's coefficients. Indeed, with the notations in  Example~\ref{Ex7.1}, we can compute the formal adjoint of the operator $\cL$ in \eqref{7.3} and we obtain
\begin{equation}\label{7.10}
\cL^*=-\pd_x^3-\pd_x\cdot u_1(x)+u_0(x)=-\pd_x^3- u_1(x)\pd_x+u^*_0(x),
\end{equation}
where $u_1(x)$ is given in \eqref{7.4} and 
\begin{equation}\label{7.11}
u^*_0(x)=u_0(x)-u_1'(x)=-\frac{3(x+s_1)((s_1+x)^2-2s_2)}{2\tau^3(x;s)}.
\end{equation}
Using \eqref{7.6} and setting
\begin{subequations}\label{7.12}
\begin{align}
\tilde{H}_1^1(x,y)&=-H_1^1(y,x)=\frac{3(xy+s_1(x+y)+s_1^2-2s_2)}{2\tau(x;s)\tau(y;s)}\\
\tilde{H}_1^0(x,y)&=H_1^0(y,x)=-\frac{3(y+s_1)}{2\tau(x;s)\tau(y;s)},
\end{align}
\end{subequations}
it follows that the heat kernel of 
$$\pd_t v =\cL^* v$$
for the operator $\cL^*$ in \eqref{7.10} is 
\begin{equation*}
\cH(x,y,t)= \frac{1}{\sqrt[3]{3t}}\Ai\left(\frac{x-y}{\sqrt[3]{3t}}\right)\left(1+\tilde{H}_1^{0}(x,y)t\right) +  \frac{1}{\sqrt[3]{9t^2}}\Ai'\left(\frac{x-y}{\sqrt[3]{3t}}\right) \tilde{H}_1^{1}(x,y)t.
\end{equation*}
We can obtain all this directly from \thref{th6.1} by observing that $\cL^*=\cK_{-z^3}$ if we work with the subspace 
$$W_s^*=\clspan\left\{\frac{s_1^2}{2}+s_2-\frac{s_1}{z}+\frac{1}{z^2},z,z^2,z^3,\dots\right\}\in \Gr_0^{(3)},$$
which is dual to $W_s$ in Example~\ref{Ex7.1} with respect to the bilinear form
$$\langle f, g\rangle =\res_z[f(z)g(-z)].$$
Note that if we set $s_1=s_2=0$ we obtain the space
$$W_0^*=\clspan\left\{\frac{1}{z^2},z,z^2,z^3,\dots\right\}\in \Gr_0^{(3)},$$
whose tau-function is the Schur polynomial $\cS_{1,1}$ corresponding to the partition $\nu=(1,1)$. We can think of $W_s^*$ as the subspace obtained from $W_0^*$ by applying appropriately  the KP flows. Since $W_s^*\in\Gr_0^{(N)}$ for every $N\geq 3$, we can obtain simple formulas for the heat kernel of the operators $\cK_{z^N}$ as we did in Remark~\ref{Re7.2}, or we can apply directly Corollary~\ref{Co4.6} to the operators $\cL_N$ in \eqref{7.8} for all $N\geq 3$.
\end{Example}


\section{Rational solutions of the Gelfand-Dickey hierarchy}\label{se8}

In this section we explain how the equivalent conditions in \thref{th6.1} characterize the rational solutions of the $N$-th Gelfand-Dickey hierarchy with coefficients $u_j$ vanishing at $x=\infty$.

Recall that for $W\in\Gr$, we can define a unique Baker function $\Psi_W(s;z)$ which depends on the KP flows $s=(s_1,s_2,\dots)$ and $z\in S^{1}$, such that 
\begin{enumerate}[(i)]
\item $\Psi_W(s;z)$ has the form 
\begin{equation}\label{8.1}
\Psi_W(s;z)=\left(1+\sum_{i=1}^{\infty}\frac{\psi_i(s)}{z^{i}}\right)e^{\sum_{j=1}^{\infty}s_jz^j},
\end{equation}
\item $\Psi_W(s;\cdot)$ belongs to $W$ for all $s$,
\end{enumerate}
see \cite[Proposition 5.1]{SW}. We identity $s_1$ and $x$, and with this convention, the stationary Baker function $\Psi_W(x,z)$ discussed in \seref{se6} can be obtained by setting $s_j=0$ for $j>1$. Using the coefficients $\psi_i(s)$ in \eqref{8.1} we define 
\begin{equation}\label{8.2}
S_W=1+\sum_{j=1}^{\infty}\psi_j(s)\pd_x^{-j},
\end{equation}
to be the formal wave operator corresponding to $W$, and we set 
\begin{equation}\label{8.3}
\cP_W=S_W\pd_xS_W^{-1}.
\end{equation}
Then $\cP=\cP_W$ in equation \eqref{8.3} is a formal pseudo-differential operator of the form
\begin{equation}\label{8.4}
\cP=\pd_x+\sum_{j=1}^{\infty}v_j(s)\pd_x^{-j},
\end{equation}
which satisfies the equations of the KP hierarchy
\begin{equation}\label{8.5}
\frac{\pd \cP}{\pd s_j}=[(\cP^j)_+,\cP].
\end{equation}
In particular, if $W\in \Gr^{(N)}$, then $\cL=\cP^N$ is an $N$-th order differential operator which solves the $N$-th Gelfand-Dickey hierarchy \eqref{5.1}. With these notations, we can formulate an  analog of the rational solutions of the KdV hierarchy \cite{AM,AMM} within the context of the Gelfand-Dickey hierarchy as follows.

\begin{Theorem}\label{th8.1}
Let $\cL=\pd_x^N+u_{N-2}(x,s_2,s_3)\pd_x^{N-2}+\cdots+u_0(x,s_2,s_3)$ be an $N$-th order differential operator whose coefficients depend on $x$, $s_2$ and $s_3$. The following conditions are equivalent.
\begin{enumerate}[\rm(i)]
\item The operator $\cL$ satisfies the Gelfand-Dickey equations \eqref{5.1} for $m=2$ and $m=3$, and its coefficients $u_k$ are rational functions of $x$ that vanish at $x=\infty$:
\begin{equation}\label{8.6}
\lim_{x\to\infty }u_{k}(x,s_2,s_3)=0,\qquad \text{ for all }\quad k=0,1,\dots,N-2.
\end{equation}
\item There exists $W\in\Gr_0^{(N)}$ such that $\cL=(\cP_W)^{N}|_{s_4=s_5=\cdots=0}$.
\end{enumerate}
\end{Theorem}
Note that condition (ii) above tells us that $\cL$ can be extended to a solution of the full Gelfand-Dickey hierarchy, and it will depend only on finitely many of the time variables $s_j$, as long as it satisfies \eqref{8.6} and the first two nontrivial equations. If we combine this with condition (iii) in \thref{th6.1}, we see that the heat kernel for the operator $\cL$ has finitely many terms if and only if $\cL$ is a  solution of the Gelfand-Dickey hierarchy, whose coefficients $u_j$ are rational functions of $x$ vanishing at $\infty$.

The proof of \thref{th8.1} can be deduced by combining the work of Krichever \cite{Krichever2} on the rational solutions of the KP equation with Wilson's description of the adelic Grassmannian $\Grad$ \cite{Wilson} which parametrizes these solutions.
Following Wilson, for $m\in\Nset_0$ and $\la\in\Cset$, we denote by $e(m,\la)$ the linear functional 
\begin{equation*}
\langle e(m, \lambda), g\rangle = g^{(m)}(\lambda)
\end{equation*}
on $\Cset[z]$ and by $\cC$  the infinite-dimensional vector space over $\Cset $, spanned by $e(m, \lambda)$. A linear functional $c\in\cC$ is called a {\it one point condition} if it involves 
derivatives at only one point $\la$, i.e. if $c$ can be written in the form
\begin{equation}\label{8.7}
c=\sum_{\text{finitely many $m$'s}} a_me(m, \la), \qquad\text{ for some }\la\in\Cset \text{ and }a_m\in\Cset.
\end{equation}
We call $\la$ the support of the condition $c$ in \eqref{8.7}. For each finite-dimensional subspace $C \subset \cC$ we set 
\begin{equation*}
V_C=\{ g\in \Cset[z]:\; \langle c,g\rangle=0 \text{ for } c \in C\}.
\end{equation*}
The {\em adelic Grassmannian} $\Grad$ consists of all spaces $W\subset \Cset(z)$ of the form
\begin{equation}\label{8.8}
W=\frac{1}{r(z)}V_C,
\end{equation} 
where 
\begin{enumerate}[(i)]
\item $C$ has a basis $\{c_1,\dots,c_n\}$ consisting of one point conditions $c_j\in\cC$, and
\item $r(z)=\prod (z-\lambda_i)^{n_i}$ is a polynomial of degree $n=\dim(C)$, where $n_i$ denotes the number 
of conditions $c_j$ supported at the point $\lambda_i$.
\end{enumerate}
If $W\in \Grad$, we can assume that its $L^2$ closure $\overline{W}$ belongs to $\Gr$ by replacing (if necessary) the unit circle in the definition of $\Gr$ with a circle of sufficiently large radius $R$, so that the supports of all conditions $c_j$ in (i) are inside the circle $|z|=R$. Conversely, we can recover $W$ from $\overline{W}$ by taking $W=\overline{W}\cap \Cset(z)$. In view of this, we can think of every $W\in \Grad$ as a point in $\Gr$, and thus we can associate a solution of the KP hierarchy. In particular, if all conditions $c_j$ in (i) above are supported at $0$, we obtain the Grassmannian $\Gr_0$ discussed in \seref{se6}, so  $\Gr_0\subset \Grad$.
With these definitions, we can summarize some of the statements in \cite{Krichever2,Shiota,Wilson} in a form suitable for the proof of \thref{th8.1} as follows.
\begin{Proposition}\label{pr8.2}
Suppose that $\cP=\pd_x+\sum_{k=1}^{\infty}v_k(x,s_2,s_3)\pd_x^{-k}$ is a formal pseudo-differential operator whose coefficients $v_k$ depend on  $x$, $s_2$ and $s_3$. The following conditions are equivalent.
\begin{enumerate}[\rm(i)]
\item $\cP$ satisfies the KP equations \eqref{8.5} for $j=2$ and $j=3$, and its coefficients $v_k$ are rational functions of $x$ that vanish at $x=\infty$.
\item There exists $W\in\Grad$ such that  $\cP=(\cP_W)|_{s_4=s_5=\cdots=0}$, where $\cP_W$ is the operator constructed from $W$ in \eqref{8.3}.
\end{enumerate}
\end{Proposition}

\begin{proof} The implication {\rm{(ii)}$\Rightarrow$\rm{(i)}} follows immediately from Wilson's work \cite{Wilson} since for $W\in\Grad$, the tau-function $\tau_W(s)$ is a polynomial in $s_1=x$ with constant leading coefficient, which implies that the coefficients $\psi_j$ are rational functions of $x$ and tend to $0$ as $x\to\infty$. Formula \eqref{8.3} shows that this is also true for the coefficients of the operator $\cP_W$.

Conversely, suppose now that (i) holds, and therefore $\cP$ satisfies the KP equations \eqref{8.5} for $j=2$ and $j=3$. It is well-known that the function 
$u(x,s_2,s_3)=2v_1(x,s_2,s_3)$ satisfies the KP equation
\begin{equation*}						
3\pd^2_{s_2}u=\left(4\pd_{s_3}u-u'''-6uu'\right)'.
\end{equation*}
where $'$ denotes the derivative with respect to $x$, and $\pd_{s_j}=\pd/\pd_{s_j}$ for $j=2,3$.
Since $u(x,s_2,s_3)$ is a rational function of $x$ and $u(x,s_2,s_3)\to 0$ as $x\to\infty$, the results in \cite{Krichever2,Shiota,Wilson} imply that there exists $W\in\Grad$ such that 
$$v_1(x,s_2,s_3)=\frac{\pd^2 \ln \tau_W(x,s_2,s_3,0,0,\dots)}{\pd x^2}. $$  
It is not hard to see now that for this $W\in\Grad$ we have $\cP=\cP_W|_{s_4=s_5=\cdots=0}$. Indeed, both $\cP$ and $\cP_W$ are formal pseudo-differential operators of the form given in \eqref{8.4} with rational coefficients vanishing at $x=\infty$. Moreover, if we consider $\tilde{\cP}=\cP_W|_{s_4=s_5=\cdots=0}$ we have 
\begin{equation*}
\tilde{\cP}=
\pd_x+\sum_{j=1}^{\infty}\tilde{v}_j(x,s_2,s_3)\pd_x^{-j}, \qquad \text{ and }\quad \tilde{v}_1(x,s_2,s_3)=v_1(x,s_2,s_3).
\end{equation*}
Using now the fact that both $\cP$ and $\tilde{\cP}$ satisfy the KP equations \eqref{8.5} for $j=2$, we can deduce by induction on $n$ that $\tilde{v}_n(x,s_2,s_3)=v_n(x,s_2,s_3)$ for all $n\in\Nset$. Indeed, we already know that this is true when $n=1$, and computing the coefficients of $\pd_x^{-\ell}$ on both sides of \eqref{8.5} for $j=2$ yields 
$$\frac{\pd v_{\ell}}{\pd s_2}=2v_{\ell+1}'+v_{\ell}''-\sum_{k=1}^{\ell-1}\binom{-k}{\ell-k}v_kv_1^{(\ell-k)}.$$
Thus, if we assume that $\tilde{v}_n(x,s_2,s_3)=v_n(x,s_2,s_3)$ for all $n\leq \ell$, and if we subtract the above equations for $\cP$ and $\tilde{\cP}$ we see that
$$\frac {\pd (v_{\ell+1} -\tilde{v}_{\ell+1} )}{\pd x}=0,$$
and therefore $v_{\ell+1}(x,s_2,s_3) -\tilde{v}_{\ell+1}(x,s_2,s_3) =c_{\ell+1}(s_2,s_3)$ is independent of $x$. But if we let $x\to\infty$ we conclude that $c_{\ell+1}(s_2,s_3)=0$, completing the inductive step, and thus the proof.
\end{proof}

\begin{Lemma}\label{le8.3}
For $N\in\Nset$ we have $\Grad\cap \Gr^{(N)}=\Gr_0^{(N)}$.
\end{Lemma}

\begin{proof}
Clearly, $\Gr_0^{(N)}\subset \Grad\cap \Gr^{(N)}$, so the nontrivial part is to establish the opposite inclusion. Suppose that $W\in  \Grad \cap \Gr^{(N)}$. By the definition of $\Grad$, there exists a finite-dimensional subspace $C \subset \cC$ which has a basis consisting of one point conditions and \eqref{8.8} holds. The main point is to show that 
\begin{equation}\label{8.9}
\text{if }\la\neq 0\quad\text{ and  }\quad c=\sum_{m=0}^{n} a_me(m, \la)\in C, \quad\text{ where }a_n\neq 0, 
\end{equation}
then 
\begin{equation}\label{8.10}
e(m, \la)\in C \qquad \text{ for all }\qquad m=0,\dots,n. 
\end{equation}
This means that we can ignore all conditions supported at $\la\neq 0$, because they simply imply that the polynomials in $V_C$ are divisible by $(z-\la)^k$, where $k$ is the number of the conditions supported at $\la$ in the basis of $C$, which will cancel with the same factor in $r(z)$. Thus, if we know that \eqref{8.9} implies \eqref{8.10} then we can assume that $C$ contains only conditions supported at $0$, which means that $W\in \Gr_0^{(N)}$. Suppose now that  \eqref{8.9} holds for some $n>0$. Since $W\in \Gr^{(N)}$ it follows that $V_C$ is invariant under the multiplication by $z^N$. Therefore, for every $g\in V_C$ we have
\begin{equation}\label{8.11}
\langle c, z^Ng(z) -\la^Ng(z) \rangle =0.
\end{equation}
It is easy to see that the last equation can be rewritten as $\langle \tilde{c}, g \rangle =0$, where 
$$\tilde{c}=\sum_{m=0}^{n-1} \tilde{a}_me(m, \la), \quad\text { and }\quad \tilde{a}_{n-1}=na_nN\la^{N-1}  \neq 0.$$
Since \eqref{8.11} holds for all $g\in V_C$, it follows that $\tilde{c}\in C$. Thus, if \eqref{8.9} holds, then $\tilde{c}\in C$ and since $\tilde{a}_{n-1}\neq 0$, we can iterate this process and prove \eqref{8.10} by induction on $n$.
\end{proof}

Since \leref{le8.3} is the key ingredient in the proof of \thref{th8.1} and in the bispectral characterization of the operators satisfying the equivalent conditions in \thref{th6.1}, we include an elegant and illuminating proof which was kindly provided by George Wilson.

\begin{proof}[Second proof of \leref{le8.3}]
Let $X$ be a framed curve; that is, its ring of functions $A$ is a subalgebra of $\Cset[z]$ and the corresponding map  $p : \Cset \to X$  is bijective. Suppose also that $A$ contains $\Cset[z^N]$ for some $N > 1$, so that now we have two maps 
\begin{equation}\label{compos}
\Cset \to X \to \Cset
\end{equation}
whose composition is  $z \mapsto z^N$. The lemma claims that in this case $p(0)$ is the only possible singular point of $X$. To see that, let's fix any other point $p(a)$ of $X$, and choose a disk $D$ containing $a$, small enough so that $z\mapsto z^N$  is an analytic isomorphism from $D$ to some small neighbourhood  $D'$ of $a^N$.  Restricting \eqref{compos} to $D$, we now have bijective analytic maps
\begin{equation}\label{analytic-maps}
D \to p(D) \to D'
\end{equation}
whose composition is an analytic isomorphism.  That implies that the inverse map to $p$ in \eqref{analytic-maps} is analytic (because it is the composition of the second map in \eqref{analytic-maps} with the inverse of the isomorphism $D \to D'$).  Thus the first map $p$ in \eqref{analytic-maps} is also an analytic isomorphism; in particular $p(D)$ contains no singular points.
\end{proof}

\begin{proof}[Proof of \thref{th8.1}]
Suppose first that (i) holds, i.e. $\cL$ is an operator whose coefficients are rational functions of $x$ satisfying \eqref{8.6}. This means that $\cP=\cL^{1/N}$ will satisfy the KP equations \eqref{8.5} for $j=2$ and $j=3$. Moreover, since the coefficients of $\cP$ are differential polynomials of the coefficients of $\cL$ with no constant terms, it follows that the coefficients of $\cP$ are also rational functions of $x$  that vanish at $x=\infty$. By Proposition~\ref{pr8.2}, there exists $W\in\Grad$ such that  $\cL^{1/N}=\cP=(\cP_W)|_{s_4=s_5=\cdots=0}$. Since $\cP^N$ is a differential operator, it follows that $W\in \Gr^{(N)}$ and by \leref{le8.3} we conclude that $W\in \Gr_0^{(N)}$, proving (ii). The implication (ii)$\Rightarrow$(i) follows easily from Proposition~\ref{pr8.2}.
\end{proof}

\section{Rank-one bispectral operators}\label{se9}

Suppose that $\cL$ is a monic differential operator that belongs to a rank-one commutative ring of differential operators. Using the notations in \seref{se6}, we know that there exists $W\in\Gr$ such that $\cL\in \cA_W$. Up to a factor depending only on $z$, the Baker function $\Psi_W(x,z)$ is the unique common eigenfunction for all operators in $\cA_W$ satisfying \eqref{6.2}. In view of the work of Duistermaat and Gr\"unbaum \cite{DG}, Wilson \cite{Wilson} proposed to call the ring $\cA_W$ {\em bispectral} if $\Psi_W(x,z)$ is also an eigenfunction of a differential operator $B(z,\pd_z)$ acting on the variable $z$, i.e.
\begin{equation*}
B(z,\pd_z) \Psi_W(x,z) =\theta(x) \Psi_W(x,z),
\end{equation*}
for some nonconstant function $\theta(x)$, and he classified all rank-one bispectral rings which are parametrized by the adelic Grassmannian $\Grad$. From his result and \leref{le8.3}, it follows easily that the sub-Grassmannian $\Gr_0$ parametrizes the operators in rank-one bispectral rings with coefficients vanishing at $\infty$.

\begin{Proposition}\label{pr9.1} 
For an $N$-th order operator $\cL=\pd_x^N+u_{N-2}(x)\pd_x^{N-2}+\cdots+u_0(x)$ the following conditions are equivalent.
\begin{enumerate}[\rm(i)]
\item $\cL$ belongs to a rank-one bispectral commutative ring and $\lim_{x\to\infty}u_{j}(x)=0$ for $j\in\{0,\dots,N-2\}$.
\item There exists $W\in\Gr_0^{(N)}$ such that $\cL=\cK_{z^N}$.
\end{enumerate}
\end{Proposition}

\begin{proof}
The implication {\rm{(ii)}$\Rightarrow$\rm{(i)}} is clear since $\Gr_0^{(N)}\subset\Grad$, hence Wilson's result \cite{Wilson} implies that $\cA_W$ is a bispectral ring, and we know from \thref{th6.1} that the coefficients $u_j(x)$ vanish at $x=\infty$. Suppose now that (i) holds. Since the operator $\cL$ is normalized so that its leading two coefficients are the constants $1$ and $0$, it follows from the results in \cite{Wilson} that the operator $\cL$ belongs to a rank-one commutative algebra built from $W\in\Grad$, i.e.  there exist $W\in\Grad$ and a monic polynomial 
$$g(z)=z^N+\be_{N-2}z^{N-2}+\cdots+\be_0\in A_W,$$ 
such that $\cL=\cK_g$. This means that
\begin{equation}\label{9.1}
\cL \Psi_W(x,z)=g(z)\Psi_W(x,z),
\end{equation}
or equivalently 
\begin{equation}\label{9.2}
\cL =S_W g(\pd_x) S_W^{-1}.
\end{equation}
The pseudo-differential operator $S_W=1+\sum_{j=1}^{\infty}\psi_j(x)\pd_x^{-j}$ in \eqref{9.2} is obtained from the one in \eqref{8.2} by setting $s_1=x$ and $s_2=s_3=\cdots=0$. Since $W\in\Grad$, the coefficients $\psi_j(x)$ are rational functions of $x$, vanishing at $\infty$. This implies that, for every $j=0,1,\dots,N-2$, the coefficient of $\pd_x^j$ on the right-hand side of \eqref{9.2} tends to $\be_j$ as $x\to\infty$, which combined with \eqref{9.2} and (i) shows that $\lim_{x\to\infty} u_{j}(x)=\be_j=0$. Therefore $g(z)=z^N$, which means that $W\in\Gr^{(N)}$. From \leref{le8.3} we see that $W\in\Gr_0^{(N)}$, completing the proof.
\end{proof}

\begin{Remark}[Darboux transformations]\label{Re9.2}
Recall that if a monic differential operator $\cL_0$ is factored as $\cL_0=\cU_0\cV_0$ where $\cU_0$ and $\cV_0$ are monic differential operators, then the operator $\cL_1=\cV_0\cU_0$ obtained by exchanging the roles of $\cU_0$ and $\cV_0$ is called a Darboux transformation from $\cL_0$. Adler and Moser \cite{AM} discovered that all rational solutions $\cL=\pd_x^2+u$ of the KdV equation with potential $u$ vanishing at $x=\infty$ can be obtained by iterating the Darboux transformation 
\begin{align}
&\cL_0=\cU_0\cV_0\leadsto \cL_1=\cV_0\cU_0=\cU_1\cV_1
\leadsto \cL_2=\cV_1\cU_1=\cU_2\cV_2\leadsto \cdots                      \nonumber\\
&\quad \leadsto \cL_{n-1}=\cV_{n-2}\cU_{n-2}=\cU_{n-1}\cV_{n-1}\leadsto \cL_n=\cV_{n-1}\cU_{n-1},  \label{9.3}
                               \end{align}
starting with  the operator $\cL_0=\pd_x^2$.  Duistermaat and Gr\"unbaum \cite{DG} showed that the corresponding operators $\cL_n$ exhaust all rank-one bispectral Schr\"odinger operators, and used the Darboux transformation to describe all bispectral second-order operators. Moreover, Darboux transformations were used to produce bispectral commutative rings of arbitrary rank \cite{BHY,KR}. Similarly to the $N=2$ case, we can characterize the operators $\cL$ satisfying the equivalent conditions in Proposition~\ref{pr9.1} as the operators obtained from $\cL_0=\pd_x^N$ by iterating the Darboux transformation \eqref{9.3}. Indeed, if we set $\cL=\cL_n$ and $\cV=\cV_{n-1}\cV_{n-2}\cdots\cV_0$, then \eqref{9.3} implies that 
\begin{equation}\label{9.4}
\cL\cV = \cV\cL_0,\quad\text{ and }\quad \cL_0:\ker(\cV)\to \ker(\cV)\text{ is nilpotent.}
\end{equation}
Conversely, one can show by induction on the order of the operator $\cV$ that \eqref{9.4} implies that the operator $\cL$ can be obtained by a sequence of Darboux transformations from $\cL_0$ as in \eqref{9.3}. Note now that for $\cL_0=\pd_x^N$ the condition that $\cL_0|_{\ker(\cV)}$ is nilpotent is equivalent to $\ker(\cV)\subset\Cset[x]$. Combining this with \leref{le6.2}, we see that the operators $\cL$ satisfying \eqref{9.4} with $\cL_0=\pd_x^N$  and for some monic differential operator $\cV$ are precisely the operators of the form $\cL=\cK_{z^N}$ built from $W\in\Gr_0^{(N)}$. 
\end{Remark}

\section*{Acknowledgments} 
I would like to thank George Wilson for providing the direct and illuminating proof of \leref{le8.3}.

\end{document}